\documentclass[12pt,reqno]{amsart}
\usepackage{amssymb}

\usepackage[all,cmtip]{xy}
\usepackage[dvips]{graphics}
\usepackage{epsfig}
\usepackage{subfigure}
\usepackage{hyperref}

\newtheorem{theorem}{Theorem}[section]
\newtheorem{proposition}[theorem]{Proposition}

\newtheorem{prob}[theorem]{Problem}
\newtheorem{lemma}[theorem]{Lemma}
\newtheorem{claim}[theorem]{Claim}
\newtheorem{corollary}[theorem]{Corollary}
\theoremstyle{definition}
\newtheorem{definition}[theorem]{Definition}

\theoremstyle{remark}
\newtheorem{remark}[theorem]{Remark}

\DeclareMathOperator{\trop}{trop}
\DeclareMathOperator{\Trop}{Trop}

\DeclareMathOperator{\Div}{Div}
\DeclareMathOperator{\divi}{div}
\DeclareMathOperator{\Prin}{Prin}
\DeclareMathOperator{\Jac}{Jac}

\DeclareMathOperator{\tr}{tr}

\DeclareMathOperator{\id}{id}
\DeclareMathOperator{\val}{val}

\DeclareMathOperator{\Stab}{Stab}

\newcommand{\mg}{\Gamma}
\newcommand{\mt}{T}

\newcommand{\R}{\mathbb R}
\newcommand{\Z}{\mathbb Z}

\DeclareMathOperator{\Aut}{Aut}

\newcommand{\Mtrg}{\ensuremath{M^{\text{tr}}_g}}

\newcommand{\sbs}{\subsection}

\newcommand{\tfae}{the following are equivalent}
\newcommand{\wrt}{with respect to}
\title{Tropical hyperelliptic curves}
\author{Melody Chan}
\email{mtchan@math.berkeley.edu}
\address{Department of Mathematics, University of California, Berkeley}

\date{\today}

\begin{document}

\begin{abstract}
We study the locus of tropical hyperelliptic curves inside the moduli space of tropical curves of genus $g$.   We define a harmonic morphism of metric graphs and prove that a metric graph is hyperelliptic if and only if it admits a harmonic morphism of degree 2 to a metric tree.  This generalizes the work of Baker and Norine on combinatorial graphs to the metric case. We then prove that the locus of 2-edge-connected genus $g$ tropical hyperelliptic curves is a $(2g-1)$-dimensional stacky polyhedral fan whose maximal cells are in bijection with trees on $g-1$ vertices with maximum valence 3.  
Finally, we show that the Berkovich skeleton of a classical hyperelliptic plane curve satisfying a certain tropical smoothness condition lies in a  maximal cell of genus $g$ called a standard ladder.
\end{abstract}

\maketitle

\tableofcontents

\section{Introduction}

In this paper, we study the locus of hyperelliptic curves inside the moduli space of tropical curves of a fixed genus $g$.  Our work ties together two strands in the recent tropical geometry literature: tropical Brill-Noether theory on the one hand \cite{spec}, \cite{capbn}, \cite{cdpr}, \cite{lpp}; and tropical moduli spaces of curves on the other \cite{bmv}, \cite{cap10}, \cite{capsurvey}, \cite{ch}, \cite{mz}.  

The work of Baker and Norine in \cite{bn07} and Baker in \cite{spec} has opened up a world of fascinating connections between algebraic and tropical curves. The Specialization Lemma of Baker \cite{spec}, recently extended by Caporaso \cite{capbn}, allows for precise translations between statements about divisors on algebraic curves and divisors on tropical curves. One of its most notable applications is to \emph{tropical Brill-Noether theory}. The classical Brill-Noether theorem in algebraic geometry, proved by Griffiths and Harris, \cite{gh}, is the following.

\begin{theorem}\label{t:bn}Suppose $g, r$, and $d$ are positive numbers and let
\[ \rho(g,r,d) = g-(r+1)(g-d+r). \]
\begin{enumerate}
	\item
If $\rho \geq 0$, then every smooth projective curve $X$ of genus $g$ has a divisor of degree $d$ and rank at least $r$. In fact, the scheme $W^r_d(X)$ parametrizing linear equivalence classes of such divisors has dimension $\min\{\rho,g\}$.
\item
If $\rho < 0$, then the general smooth projective curve of genus $g$ has no divisors of degree $d$ and rank at least $r$.
\end{enumerate}
\end{theorem}

A tropical proof of the Brill-Noether theorem by way of the Specialization Lemma was conjectured in \cite{spec} and obtained by Cools, Draisma, Payne, and Robeva in \cite{cdpr}.  See \cite{capbn} and \cite{lpp} for other advances in tropical Brill-Noether theory.

Another strand in the literature concerns the $(3g-3)$-dimensional moduli space $\Mtrg$ of tropical curves of genus~$g$.  This space was considered by Mikhalkin and Zharkov in \cite{mapplications} and \cite{mz}, in a more limited setting (i.e.~without vertex weights).  It was constructed and studied as a topological space by Caporaso, who proved that $\Mtrg$ is Hausdorff and connected through codimension one \cite{cap10}, \cite{capsurvey}.   In \cite{bmv}, Brannetti, Melo, and Viviani constructed it explicitly in the category of stacky polyhedral fans (see \cite[Definition 3.2]{ch}).
In \cite{ch}, we gained a detailed understanding of the combinatorics of $\Mtrg$, which deeply informs the present~study.

Fix $g, r,$ and $d$ such that $\rho(g,r,d) < 0$. Then the \emph{Brill-Noether locus} $\mathcal{M}^r_{g,d}\subset \mathcal{M}_g$ consists of those genus $g$ curves which are exceptional in Theorem \ref{t:bn}(ii) in the sense that they do admit a divisor of degree $d$ and rank at least $r$. The tropical Brill-Noether locus ${M}^{r,\tr}_{g,d}\subset \Mtrg$ is defined in exactly the same way.

In light of the recent advances in both tropical Brill-Noether theory and tropical moduli theory, it is natural to pose the following

\begin{prob}\label{p:bn}
Characterize the tropical Brill-Noether loci $M_{g,d}^{r,tr}$ inside $M_g^{tr}$.
\end{prob}
\noindent The case $r=1$ and $d=2$ is, of course, the case of hyperelliptic curves, and the combinatorics is already very rich.  In this paper, we are able to characterize the hyperelliptic loci in each genus.  The main results of this paper are the follwing three theorems, proved in Sections 3, 4, and 5, respectively.

\begin{theorem}
\label{t:mainintro}
Let $\mg$ be a metric graph with no points of valence 1, and let $(G,l)$ denote its canonical loopless model. Then the following are equivalent:
\begin{enumerate}
\item $\mg$ is hyperelliptic.
\item There exists an involution $i: G \to G$ such that $G/i$ is a tree.  
\item There exists a nondegenerate harmonic morphism of degree 2 from $G$ to a tree, or $|V(G)|=2$. (See Figure \ref{f:main}).
\end{enumerate}
\end{theorem}
\begin{theorem}
Let $g\ge 3$.  The locus of 2-edge-connected genus $g$ tropical hyperelliptic curves is a $(2g-1)$-dimensional stacky polyhedral fan whose maximal cells are in bijection with trees on $g-1$ vertices with maximum valence 3.   (See Figure \ref{f:h23}).
\end{theorem}
\begin{theorem}
\label{t:skeletonintro}
Let $X \subseteq \mathbb{T}^2$ be a hyperelliptic curve of genus $g \geq 3$ over a complete nonarchimedean valuated field $K$, and suppose $X$ is defined by a polynomial of the form $P=y^2+f(x)y+h(x)$. Let $\widehat{X}$ be its smooth completion.  Suppose the Newton complex of $P$ is a unimodular subdivision of the triangle with vertices $(0,0), (2g+2,0),$ and $(0,2)$, and suppose that the core of $\Trop X$ is bridgeless. Then the skeleton $\Sigma$ of the Berkovich analytification  $\widehat{X}^{\operatorname{an}}$ is a standard ladder of genus $g$ whose opposite sides have equal length.  (See Figure~\ref{f:newton}).
\end{theorem}

We begin in Section 2 by giving new definitions of harmonic morphisms and quotients of metric graphs, and recalling the basics of divisors on tropical curves. In Section 3, we prove the characterization of hyperellipticity stated in Theorem \ref{t:mainintro}.  This generalizes a central result of \cite{bn} to metric graphs.  See also \cite[Proposition 45]{hmy} for a proof of one of the three parts.
In Section 4, we build the space of hyperelliptic tropical curves and the space of 2-edge-connected hyperelliptic tropical curves. 
We then explicitly compute the hyperelliptic loci in $\Mtrg$ for $g=3$ and $g=4$. See Figures \ref{f:h3} and \ref{f:h23}. Note that all genus 2 tropical curves are hyperelliptic, as in the classical case.  
Finally, in Section 5, we establish a connection to embedded tropical curves in the plane, of the sort shown in Figure \ref{f:newton}, and prove Theorem \ref{t:skeletonintro}.

Our work in Section 5 represents a first step in studying the behavior of hyperelliptic curves under the map
\[ \trop: M_g(K) \to \Mtrg \]
in \cite[Remark 5.51]{bpr}, which sends a curve $X$ over a algebraically closed, complete nonarchimedean field $K$ to its Berkovich skeleton, or equivalently to the dual graph, appropriately metrized, of the special fiber of a semistable model for $X$.  Note that algebraic curves that are not hyperelliptic can tropicalize to hyperelliptic curves \cite[Example 3.6]{spec}, whereas tropicalizations of hyperelliptic algebraic curves are necessarily hyperelliptic \cite[Corollary 3.5]{spec}. We conjecture that the locus of hyperelliptic algebraic curves of genus $g$ maps surjectvely onto the tropical hyperelliptic locus.  It would be very interesting to study further the behavior of hyperelliptic loci, and higher Brill-Noether loci, under the tropicalization map above.

\medskip

\noindent {\bf Acknowledgments.}  The author thanks M.~Baker, J.~Rabinoff, and B.~Sturmfels for helpful comments, J.~Harris for suggesting Problem~\ref{p:bn}, R.~Masuda and J.~Rodriguez for much help with typesetting, and Sherrelle at Southwest Airlines for reuniting her with her notes for this paper.  The author is supported by a Graduate Research Fellowship from the National Science Foundation.

\begin{figure}
\includegraphics[height=1.5in]{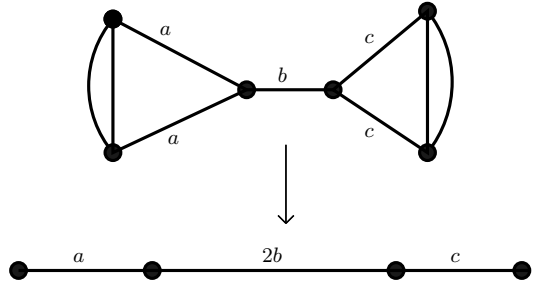}%
\caption{A harmonic morphism of degree two.  Here, $a,b$, and $c$ are positive real numbers.}
\label{f:main}
\end{figure}

\section{Definitions and notation}

We start by defining harmonic morphisms, quotients, divisors, and rational functions on metric graphs.  These concepts will be used in Theorem \ref{t:main} to characterize hyperellipticity.

\sbs{Metric graphs and harmonic morphisms}
Throughout, all of our graphs are connected, with loops and multiple edges allowed.
A {\bf metric graph} is a metric space $\mg$ such that there exists a graph $G$ and a length function $$l:E\left(G\right)\rightarrow\mathbb{R}_{>0}$$
so that $\mg$ is obtained from $(G,l)$ by gluing intervals $\left[0,l\left(e\right)\right]$ for $e\in E\left(G\right)$ at their endpoints, as prescribed by the combinatorial data of $G$.  The distance $d\left(x,y\right)$ between two points $x$ and $y$ in $\mg$ is given by the length of the shortest path between them.
We say that $(G,l)$ is a {\bf model} for $\mg$ in this case.  We say that $(G,l)$ is a loopless model if $G$ has no loops.  
Note that a given metric graph $\mg$ admits many possible models $(G,l)$.  For example, a line segment of length $a$ can be subdivided into many edges whose lengths sum to $a$.
The reason for working with both metric graphs and models is that metric graphs do not have a preferred vertex set, but choosing some vertex set is convenient for making combinatorial statements. 

Suppose $\mg$ is a metric graph, and let us exclude, once and for all, the case that $\mg$ is homeomorphic to the circle $S^1$.  We define the {\bf valence} $\val(x)$ of a point $x\in\mg$ to be the number of connected components in $U_x\setminus\{x\}$ for any sufficiently small neighborhood $U_x$ of $x$.  Hence almost all points in $\mg$ have valence 2.  By a {\bf segment} of $\mg$ we mean a subset $s$ of $\mg$ isometric to a real closed interval, such that any point in the interior of $s$ has valence 2.

If $V\subseteq \mg$ is a finite set which includes all points of $\mg$ of valence different from 2, then define a model $(G_V,l)$ as follows.  The vertices of the graph $G_V$ are the points in $V$, and the edges of $G_V$ correspond to the connected components of $\mg\setminus V$. These components are necessarily isometric to open intervals, the length of each of which determines the function $l:E\left(G_V\right)\rightarrow\mathbb{R}_{>0}$.  Then $(G_V,l)$ is a model for $\mg$.

The {\bf canonical model} $(G_0,l)$ for a metric graph $\mg$ is the model obtained by taking
$$V=\{x\in\mg~:~\val(x)\ne 2\}.$$
The {\bf canonical loopless model} $(G_-,l)$ for the metric graph $\mg$ is the model obtained from the canonical model by placing an additional vertex at the midpoint of each loop edge.  Thus, if a vertex $v\in V(G_-)$ has degree 2, then the two edges incident to $v$ necessarily have the same endpoints and the same length.

Suppose $(G,l)$ and $(G',l')$ are loopless models for metric graphs $\mg$ and $\mg'$, respectively.  A {\bf morphism of loopless models} $\phi:(G,l)\rightarrow (G',l')$ is a map of sets
$$V\left(G\right)\cup E\left(G\right)\overset{\phi}{\rightarrow}V\left(G'\right)\cup E\left(G'\right)$$
such that
\begin{enumerate}
	\item $\phi(V(G))\subseteq V(G')$,
	\item if $e=xy$ is an edge of $G$ and $\phi\left(e\right)\in V\left(G'\right)$ then $\phi\left(x\right)=\phi\left(e\right)=\phi\left(y\right)$,
	\item if $e=xy$ is an edge of $G$ and $\phi(e)\in E\left(G'\right)$ then $\phi(e)$ is an edge between $\phi(x)$ and $\phi(y)$, and
	\item if $\phi\left(e\right)=e'$ then $l'\left(e'\right)/l\left(e\right)$ is an integer.
\end{enumerate}
For simplicity, we will sometimes drop the length function from the notation and just write $\phi:G\rightarrow G'$.  An edge $e\in E(G)$ is called {\bf horizontal} if $\phi(e)\in E\left(G'\right)$ and {\bf vertical} if $\phi\left(e\right)\in V\left(G'\right)$.

Now, the map $\phi$ induces a a map $\tilde{\phi}: \mg\rightarrow\mg'$ of topological spaces in the natural way. In particular, if $e\in E\left(G\right)$ is sent to $e'\in E\left(G'\right)$, then we declare $\tilde{\phi}$ to be linear along $e$.  Let
$$\mu_\phi(e) = l'(e')/l(e)\in\Z$$
denote the slope of this linear map.

A morphism of loopless models $\phi:(G,l)\rightarrow (G',l')$ is said to be {\bf harmonic} if for every $x\in V\left(G\right)$,
the nonnegative integer
$$m_\phi(x)  =\sum_{\substack{e\in E(G) \\ x\in e,~\phi(e)=e'}} \!\!\! \mu_\phi(e)$$
is the same over all choices of $e'\in E\left(G'\right)$ that are incident to the vertex $\phi(x)$.  The number $m_\phi(x)$ is called the {\bf horizontal multiplicity} of $\phi$ at $x$.
We say that $\phi$ is {\bf nondegenerate} if $m_{\phi}\left(x\right)>0$ for all $x\in V(G)$.
The {\bf degree} of 
$\phi$ is defined to be
$$\deg\phi=\sum_{\substack{e\in E(G) \\ \phi(e)=e'}}\mu_\phi(e)$$
for any $e'\in E\left(G'\right)$.
One can check that the number $\deg\phi$ does not depend on the choice of $e'$ (\cite[Lemma 2.4]{bn}, \cite[Lemma 2.12]{ura}). If $G'$ has no edges, then we set $\deg\phi=0$.

We define a {\bf morphism of metric graphs} to be a continuous map $\tilde{\phi}: \mg\rightarrow\mg'$ which is induced from a morphism $\phi:(G,l)\rightarrow (G',l')$ of loopless models, for some choice of models $(G,l)$ and $(G',l')$.  It is {\bf harmonic} if $\phi$ is harmonic as a morphism of loopless models, and we define its degree $\deg(\tilde{\phi})$ = $\deg(\phi)$.  For any point $x\in \mg$, we define the {\bf horizontal multiplicity} of $\tilde{\phi}$ at $x$ as
$$m_{\tilde{\phi}}(x)=\begin{cases}
m_\phi(x) &\textrm{ if } x\in V(G)\\
0 &\textrm{ if } x\in e^\circ \textrm{ and }\phi(e)\in V(G')\\
\mu_\phi(e) &\textrm{ if } x\in e^\circ \textrm{ and }\phi(e)\in E(G').
\end{cases}$$
Here, $e^\circ$ denotes the interior of $e$. 
One can check that the harmonicity of $\tilde{\phi}$ is independent of choice of $\phi$, as are the computations of $\deg(\tilde{\phi})$ and $m_\phi(x)$.

\begin{figure}[t]
\includegraphics[width=\columnwidth,angle=0,scale=1.1]{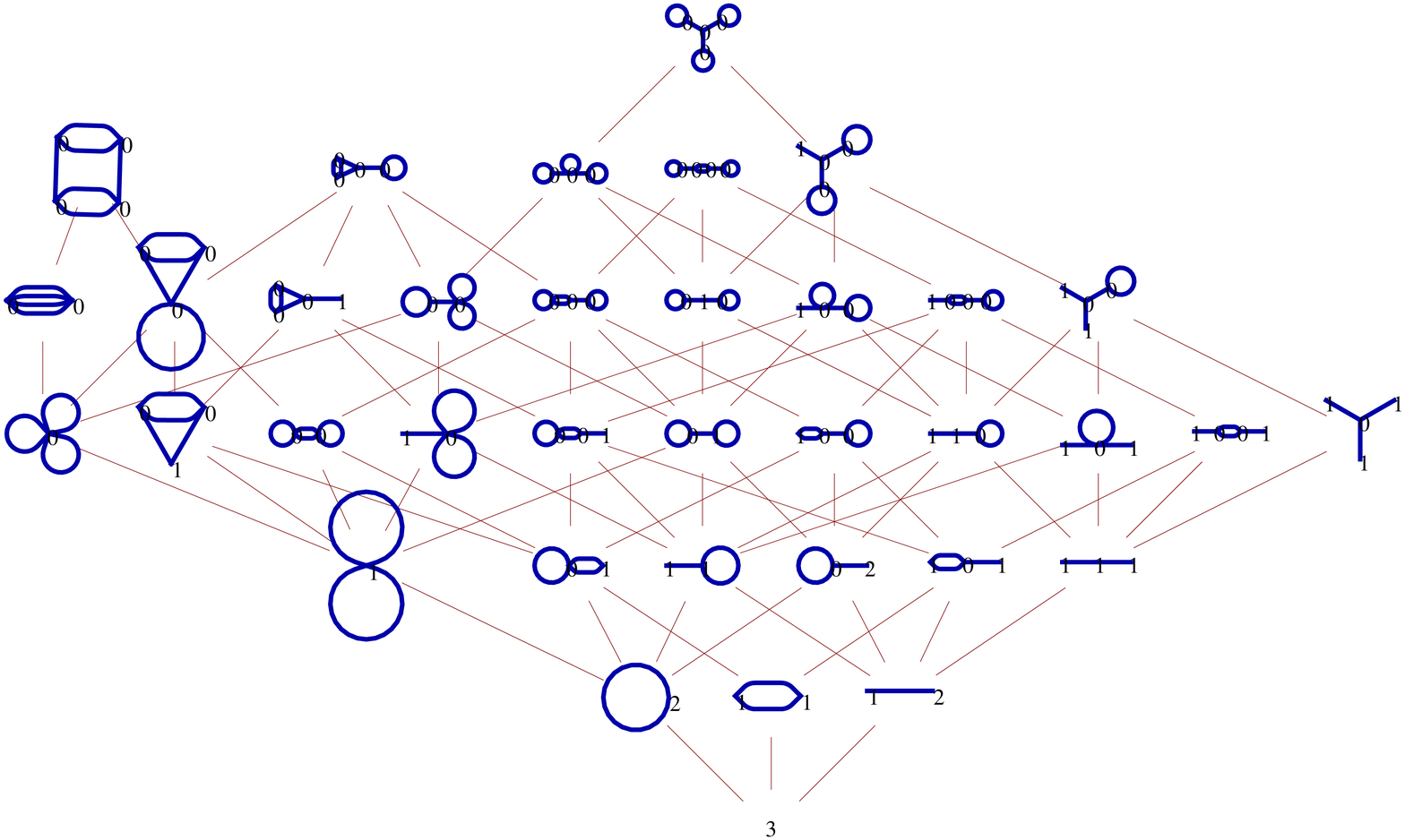}%
\vskip -.6in
  \caption{The tropical hyperelliptic curves of genus 3. Here, edges that form a 2-edge-cut are required to have the same length.  The space of such curves sits inside the moduli space $M^{tr}_3$ shown in \cite[Figure 1]{ch}.}
  \label{f:h3}
\end{figure}

\begin{figure}
\includegraphics[height=5in,angle=0,scale=1.1]{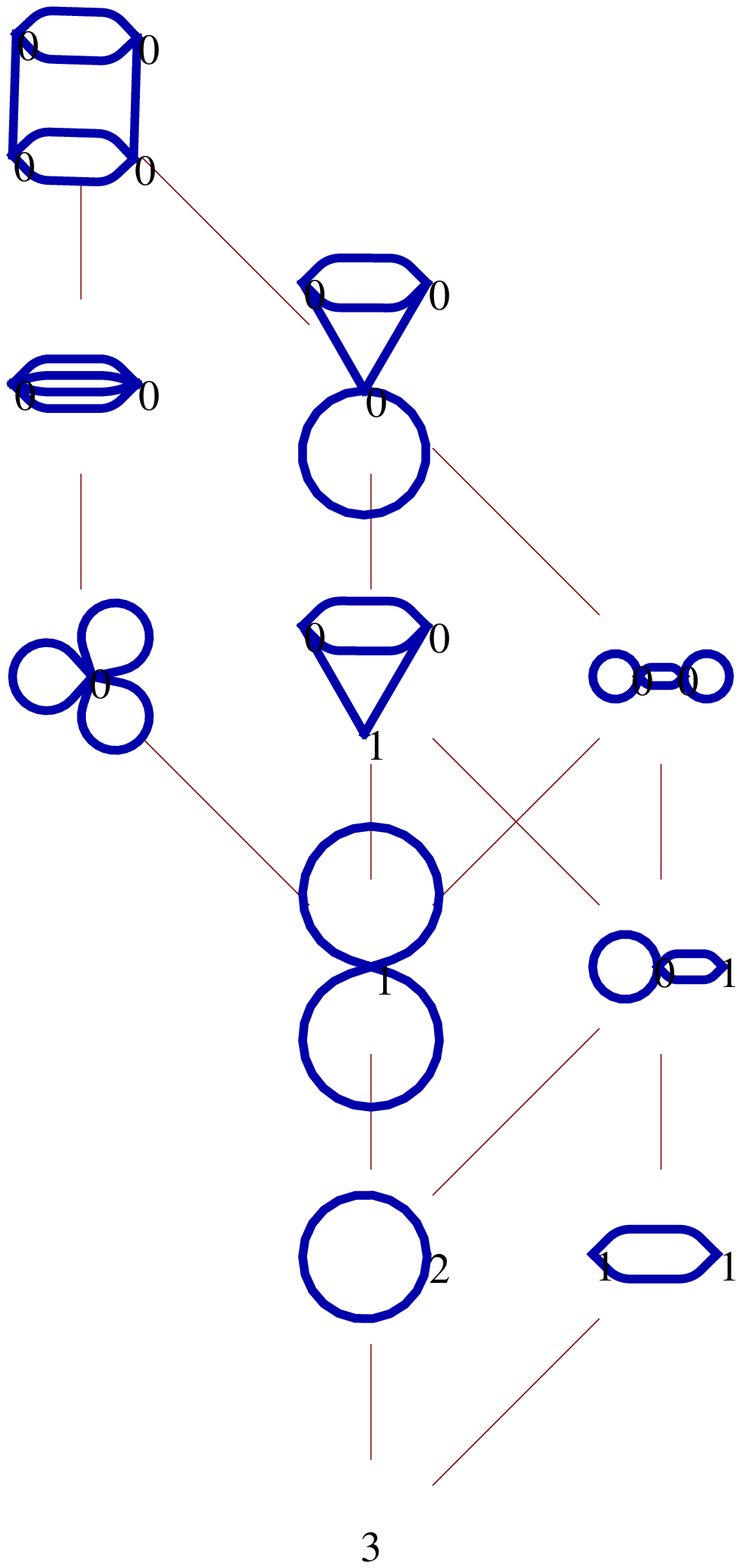}%
\vskip -.6in
  \caption{The 2-edge-connected tropical hyperelliptic curves of genus 3. Here, edges that form a 2-edge-cut are required to have the same length. The space of such curves sits inside the moduli space $M^{tr}_3$ shown in \cite[Figure 1]{ch}.}
  \label{f:h23}
\end{figure}

\sbs{Automorphisms and quotients}
Let $(G,l)$ be a loopless model for a metric graph $\mg$.  An {\bf automorphism} of $(G,l)$ is a harmonic morphism $\phi:(G,l)\rightarrow (G,l)$ of degree 1 such that the map $V\left(G\right)\cup E\left(G\right)\rightarrow V\left(G\right)\cup E\left(G\right)$ is a bijection.  Equivalently, an automorphism of $(G,l)$ is an automorphism of $G$ that also happens to preserve the length function $l$.  The group of all such automorphisms is denoted $\Aut(G,l)$.

An automorphism of $\mg$ is an isometry $\mg\rightarrow\mg$.  We assume throughout that $\mg$ is not a circle, so the automorphisms $\Aut(\mg)$ of $\mg$ form a finite group, since they permute the finitely many points of $\mg$ of valence $\ne2$.  Then automorphisms of $\mg$ correspond precisely to automorphisms of the canonical loopless model $(G_-,l)$ for $\mg$.
An automorphism $i$ is an {\bf involution} if $i^{2}=\id$.

Let $(G,l)$ be a loopless model, and let $K$ be a subgroup of $\Aut(G,l)$. We will now define the {\bf quotient loopless model} $(G/K,l')$.  The vertices of $G/K$ are the $K$-orbits of $V\left(G\right)$, and the edges of $G/K$ are the $K$-orbits of those edges $xy$ of $G$ such that $x$ and $y$ lie in distinct $K$-orbits. If $e\in E(G)$ is such an edge, let $[e]\in E(G/K)$ denote the edge corresponding to its $K$-orbit.  
The length function $$l':E\left(G/K\right)\rightarrow\mathbb{R}_{>0}$$
is given by 
$$l'\left(\left[e\right]\right)=l\left(e\right)\cdot|\Stab(e)|$$ 
for every edge $e\in E\left(G\right)$ with $K$-inequivalent ends. Notice that $l$ is well-defined on $E\left(G/K\right)$.  Note also that $(G/K,l')$ is a loopless model.  

\begin{remark}
If $K$ is generated by an involution $i$ on $(G,l)$, then computing the length function $l'$ on the quotient $G/K$ is very easy.  By definition, edges of $G$ in orbits of size 2 have their length preserved, edges of $G$ flipped by $i$ are collapsed to a vertex, and edges of $G$ fixed by $i$ are stretched by a factor of 2.  See 
Figure~\ref{f:main}.
\end{remark}

The definition of the quotient metric graph follows the definition in \cite{bn} in the case of nonmetric graphs, but has the seemingly strange property that edges can be stretched by some integer factor in the quotient.  The reason for the stretching is that it allows the natural quotient morphism to be harmonic.  Indeed, let $(G,l)$ be a loopless model, let $K$ be a subgroup of $\Aut(G,l)$, and let  $(G/K,l')$ be the quotient.  Define a morphism of loopless models
$$\pi_{K}:(G,l)\rightarrow (G/K,l')$$
as follows.  If $v\in V(G)$, let $\pi_K(v) = [v]$.  If $e\in E(G)$ has $K$-equivalent ends $x$ and $y$, then let $\pi_K(e)=[x]=[y]$.   If $e\in E(G)$ has $K$-inequivalent ends, then let $\pi_K(e)=[e]$.

\begin{lemma}\label{l:qharmonic}
Let $(G,l)$ be a loopless model, let $K$ be a subgroup of $\Aut(G,l)$, and let  $(G/K,l')$ be the quotient. Then the quotient morphism $\pi_{K}:(G,l)\rightarrow (G/K,l')$ constructed above is a harmonic morphism.  If the graph $G/K$ does not consist of a single vertex, then $\pi_K$ is nondegenerate of degree 
$|K|$.
\end{lemma}

\begin{proof}
Let $x\in V(G)$. Then for all $e'\in E(G/K)$ incident to $\pi_K(x)$, we have
$$\sum_{\substack{e\in E(G) \\ x\in e,~\phi(e)=e'}} \!\!\!\!\!\! \mu_\phi(e)
= \!\!\! \sum_{\substack{e\in E(G) \\ x\in e,~\phi(e)=e'}} \!\!\! |Stab(e)|
= |Stab(x)|$$
is indeed independent of choice of $e'$.  The last equality above follows from applying the Orbit-Stabilizer formula to the transitive action of $Stab(x)$ on the set $\{e\in E(G)~:~ x\in e,~\phi(e)=e'\}$. Furthermore, for any $e'\in E(G/K)$, we have
$$\deg(\pi_K) = \sum_{\substack{e\in E(G) \\ \phi(e)=e'}}  \frac{l'(e')}{l(e)}
= \sum_{\substack{e\in E(G) \\ \phi(e)=e'}}  |Stab(e)|
= |K|,$$
where the last equality again follows from the Orbit-Stabilizer formula.

Finally, suppose the graph $G/K$ is not a single vertex.  Then every $x'\in V(G/K)$ is incident to some edge $e'$.  Therefore, any $x\in V(G)$ with $\pi_K(x)=x'$ must be incident to some $e\in E(G)$ with $\pi_K(e)=e'$, which shows that $m_{\pi_K}(x)\ne 0$.  Hence $\pi_K(x)$ is nondegenerate.
\end{proof}

Figure~\ref{f:main} shows an example of a quotient morphism which is harmonic by Lemma \ref{l:qharmonic}. Here, the group $K$ is generated by the involution which flips the loopless model $G$ across its horizontal axis.  The quotient $G/K$ is then a path.  We will see in Theorem \ref{t:main} that $G$ is therefore hyperelliptic.

\sbs{Divisors on metric graphs and tropical curves}

The basic theory of divisors on tropical curves that follows is due to \cite{bn07} in the non-metric case and to \cite{gk}, \cite{mz} in the metric case. Let $\mg$ be a metric graph.
Then the {\bf divisor group} $\Div(\mg)$ is the group of formal $\mathbb{Z}$-sums of points of $\mg$.~If
$$D=\sum_{x\in \mg}D(x)\cdot x,\, D(x)\in\mathbb{Z}$$
is an element of $\Div(\mg)$,
then define the {\bf degree} of $D$ to be
$$\deg D=\sum_{x\in \mg}D(x)\in\Z.$$
Denote by $\Div^{0}(\mg)$ the subgroup of divisors of degree 0.

A {\bf rational function} on $\mg$ is a function $f:\mg\rightarrow\mathbb{R}$ that is continuous and piecewise-linear, with integer slopes along its domains of linearity. If $f$ is a rational function on $\mg$, then define the divisor $\divi f$ associated to it as follows: at any given point $x\in \mg$, let $(\divi f)(x)$ be the sum of all slopes that $f$ takes on along edges emanating~from~$x$.

The group of {\bf principal divisors} is defined to be $$\Prin(\mg)=\{ \divi f~:~ f\text{ a rational function on }\mg\}.$$
One may check that $\Prin(\mg)\subseteq\Div^{0}(\mg)$.
Then the {\bf Jacobian} $\Jac(\mg)$ is defined to be $$\Jac(\mg)=\frac{\Div^{0}(\mg)}{\Prin(\mg)}.$$
Note that $\Jac(\mg)=1$ if and only if $\mg$ is a tree; this follows from \cite[Corollary 3.3]{bf}.

A divisor $D\in\Div (\mg)$ is {\bf effective}, and we write $D\geq0$, if $D=\sum_{x\in \mg}D(x)\cdot x$ with $D(x)\in\mathbb{Z}_{\geq0}$ for all $x\in X$.
We say that two divisors $D$ and $D'$ are {\bf linearly equivalent}, and we write $D\sim D'$, if $D-D'\in\Prin(\mg)$.  Then
the {\bf rank} of a divisor $D$ is defined to be
$$\max\{ k\in\Z\,:\,\text{for all }E\geq0\text{ of degree $k$, there exists }E'\geq0 \text{ such that }D-E\sim E'\,\}.$$

\begin{definition}
A metric graph is {\bf hyperelliptic} if it has a divisor of degree $2$ and rank $1$.
\end{definition}

Let us extend the above definition to abstract tropical curves. The definition we will give is due to \cite{abc}.  First, we recall:

\begin{definition}
An {\bf abstract tropical curve} is a triple $(G,w,l)$, where $(G,l)$ is a model for a metric graph $\mg$, and
$$w:V(G)\rightarrow\mathbb{Z}_{\geq0}$$
is a weight function satisfying a stability condition as follows: every vertex $v$ with $w(v)=0$ has valence at least $3$.
\end{definition}

Now given a tropical curve $(G,w,l)$ with underlying metric graph $\mg$, let $\mg^{w}$ be the metric graph obtained from $\mg$ by adding at each vertex $v\in V(G)$ a total of $w(v)$ loops of any lengths.

\begin{definition}\label{d:thyp}
An abstract tropical curve  $(G,w,l)$ with underlying metric graph $\mg$ is {\bf hyperelliptic} if $\mg^w$ is hyperelliptic, that is, if $\mg^w$ admits a divisor of degree $2$ and rank $1$.
\end{definition}

\begin{remark}
We will see in Theorem \ref{t:main} that the definition of hyperellipticity does not depend on the lengths of the loops that we added to $\mg$.
\end{remark}

\sbs{Pullbacks of divisors and rational functions}

Let $\phi:\mg\rightarrow \mg'$ be a harmonic morphism, and let $g:\mg'\rightarrow\mathbb{R}$ be a rational function.  The {\bf pullback} of $g$ is the function $\phi^{*}g:\mg\rightarrow\mathbb{R}$ defined by $$\phi^{*}g=g\circ\phi.$$  One may check that  $\phi^{*}g$ is again a rational function.
The {\bf pullback} map on divisors
$$\phi^{*}:\Div \mg'\rightarrow\Div \mg$$
is defined as follows: given $D'\in\Div\mg'$, let $$\left(\phi^{*}\left(D'\right)\right)\left(x\right)=m_{\phi}\left(x\right)\cdot D'\left(\phi(x)\right)$$ for all $x\in \mg$.
The following extends \cite[Proposition 4.2(ii)]{bn} to metric graphs.  It will be used in the proof of Theorem \ref{t:2ec}.

\begin{proposition} \label{p:pushpull}Let $\phi:\mg\rightarrow \mg'$ be a harmonic morphism of metric graphs, and let $g:\mg'\rightarrow\mathbb{R}$ be a rational function. Then $$\phi^{*}\divi g  =  \divi\phi^{*}g.$$
\end{proposition}

\begin{remark}
The two occurrences of $\phi^{*}$ are really different operations, one on divisors and the other on rational functions.
Also, note that there is an analogous statement for pushforwards along $\phi$, as well as many other analogues of standard properties of holomorphic maps, but we do not need them in this paper.
\end{remark}

\begin{proof} It is straightforward (and we omit the details) to show that we may break $\mg$ and $\mg'$ into sets $S$ and $S'$ of segments along which $\phi^*g$ and $g$, respectively, are linear, and furthermore, such that each segment $s\in S$ is mapped linearly to some $s'\in S$ or collapsed to a point.
 Then at any point $x'\in \mg$, we have
$$(\divi g)(x')=\sum_{s'=x'y'\in S'}\frac{g(y')-g(x')}{l(s')}.$$
So for any $x\in \mg$,
\begin{eqnarray*}
(\phi^{*}\divi g)(x)
&=& m_{\phi}(x)\sum_{s'=\phi(x)y'\in S'}\frac{g(y')-g(\phi(x))}{l(s')} \\
&=& \sum_{s'=\phi(x)y'\in S'}    \sum_{s=xy\in\phi^{-1}(s')} \frac{l(s')}{l(s)} \cdot\frac{g(y')-g(\phi(x))}{l(s')} \\ \\
&=&\sum_{s'=\phi(x)y'\in S'}    \sum_{s=xy\in\phi^{-1}(s')} \frac{g(\phi(y))-g(\phi(x))}{l(s)} \\ \\
&=&(\divi\phi^{*}g)(x).
\end{eqnarray*}
\vskip -.75cm
\end{proof}

\section{When is a metric graph hyperelliptic?}

Our goal in this section is to prove Theorem \ref{t:main}, which characterizes when a metric graph hyperelliptic.  It is the metric analogue of one of the main theorems of \cite{bn}.  We will first prove the theorem for 2-edge-connected metric graphs in Theorem \ref{t:2ec} and then prove the general case.  

Given $f$ a rational function on a metric graph $\mg$, let $M\left(f\right)$ (respectively $m\left(f\right)$) denote the set of points of $\mg$ at which $f$ is maximized (respectively minimized).
Recall that a graph is said to be {\bf $k$-edge-connected} if removing any set of at most $k-1$ edges from it yields a connected graph.  A metric graph $\mg$ is said to be $k$-edge-connected if the underlying graph $G_0$ of its canonical model $(G_0,l)$ is $k$-edge-connected.  Thus a  metric graph is $2$-edge-connected if and only if any model for it is 2-edge-connected, although the analogous statement is false for $k>2$.

The next lemma studies the structure of tropical rational functions with precisely two zeroes and two poles.  It is important to our study of hyperelliptic curves because if $\mg$ is a graph with a degree two harmonic morphism $\phi$ to a tree $T$ and $D=\tilde{y}-\tilde{x}\in\Prin T$, then the pullback $\phi^*(D)$ is a principal divisor on $\mg$ with two zeroes and two poles by Proposition~\ref{p:pushpull}.

\begin{lemma}\label{l:slopes}
Let $\mg$ be a $2$-edge connected metric graph; let $f$ be a rational function on $\mg$ such that
$$\divi f=y+y'-x-x'$$
for $x,x',y,y'\in \mg$
such that the sets $\{ x,x'\}$
and $\{ y,y'\}$  are disjoint. Then
\begin{enumerate}
 \item $\partial M(f)=\{ x,x'\}$ and $\partial m (f)=\{ y,y'\}$, where $\partial$ denotes the boundary.  In fact, there are precisely two edges leaving $M(f)$, one at $x$ and one at $x'$, and precisely two edges leaving $m(f)$, one at $y$ and one at $y'$.  Furthermore, there is an $x{-}x'$ path in $M(f)$, and a $y{-}y'$ path in $m(f)$.
 \item  $f$ never takes on slope greater than 1. Furthermore, for any $w\in \mg$, let $s^{+}(w)$, (resp.~$s^{-}(w)$) denote the total positive (resp.~negative) outgoing slope from $w$.  Then \[s^{+}(w)\leq 2\text{ and }s^{-}(w)\geq-2.\]
\item  Given $w\in \mg$ with $(\divi f)(w)=0$, the multiset of outgoing nonzero slopes at $w$ is $\emptyset$, $\{ 1,-1\}$, or $\{ 1,1,-1,-1\}$.
\end{enumerate}
\end{lemma}

\begin{proof}
Let $D=\divi f$.
 Note that $D<0$ at any point in $\partial M(f)$, so $\partial M(f)\subset\{ x,x'\}$. Now, $\partial M(f)$ must be nonempty since $\divi f\neq0$. If $|\partial M(f)|=1$, then since $M(f)$ has at least two edges leaving it (since $\mg$ is $2$-edge-connected), and $f$ is decreasing along them, then we must have $x=x'$ and $\partial M(f) = \{x\}$ as desired. Otherwise, $|\partial M(f)|=2$ and again $\partial M(f)=\{ x,x'\}$. Now, $f$ must decrease along any segment leaving $M(f)$, and there are at least two such segments by 2-edge-connectivity of $\mg$.  By inspecting $D$, we conclude that there exactly two such segments, one at $x$ and one at $x'$.  Furthermore, suppose $x\ne x'$.  Then deleting the segment leaving $M(f)$ at $x$ cannot separate $x$ from $m(f)$, again by 2-edge-connectivity, so there is an $x{-}x'$ path in $M(f)$.  The analogous results hold for $m(f)$. This proves (i).

Let us prove (ii).  Pick any $w\in \mg$. We will show that $s^+(w)\le 2$, and moreover, that if $s^+(w)=2$ then there are precisely two directions at $w$ along which $f$ has outgoing slope $+1$.  An analogous argument holds for $s^-(w)$, so (ii) will follow.

Let $U$ be the union of all paths in $\mg$ that start at $w$ and along which $f$ is nonincreasing. Let $w_{1},\dots,w_{l}$ be the set of points in $U$ that are either vertices of $\mg$ or are points at which $f$ is not differentiable.  Let $W=\{w,w_{1},\dots,w_{l}\}$. Then $U\backslash W$ consists of finitely many open segments. Let $S=\{s_{1},\dots,s_{k}\}$ be the set of closures of these segments.  So the closed segments $s_{1},\dots,s_{k}$ cover $U$ and intersect at points in $W$, and $f$ is linear along each of them.  Orient each segment in $S$ for reference, and for each $y\in W$, let
\begin{align*}
\delta^+(y)&=\{j~:~s_j\text{ is outgoing at }y\},\\
\delta^-(y)&=\{j~:~s_j\text{ is incoming at }y\}.
\end{align*}
Finally, for each $i=1,\ldots,k$, let $m_i\in\Z$ be the slope that $f$ takes on along the (oriented) segment $s_i$.

The key observation regarding $U$ is that the slope of $f$ along any edge $e$ leaving $U$ must be positive, because otherwise $e$ would lie in $U$.
Therefore we have
\begin{align}
D(w_{i})\,&\geq\sum_{j\in\delta^+(w_i)}m_{j}-\sum_{j\in\delta^-(w_i)}m_{j}\qquad\quad\text{for }i=1,\ldots,l,\label{e:wi}\\
D(w)\,&\geq\sum_{j\in\delta^+(w)}m_{j}-\sum_{j\in\delta^-(w)}m_{j} - s^+(w).\label{e:w}
\end{align}
Summing \eqref{e:wi} and \eqref{e:w}, we have
\begin{align}
D(w)+D(w_{1})+\cdots+D(w_{l})\geq s^+(w). \label{e:sum}
\end{align}
Since $D=y+y'-x-x'$, we must have $s^+(w)\le 2$.

Suppose $s^+(w)= 2$.  Then the inequality \eqref{e:sum} must be an equality, and so \eqref{e:wi} and \eqref{e:w} must also be equalities.  In particular, no segment leaves $U$ except at $w$.  Since $\mg$ is $2$-edge-connected and $\partial U=\{w\}$, it follows that there are precisely two directions at $w$ along which $f$ has outgoing slope $+1$.  This proves (ii).

Finally, part (iii) follows directly from (ii).
\end{proof}

We now come to the metric version of \cite[Theorem 5.12]{bn}, which gives equivalent characterizations of hyperellipticity for $2$-edge-connected metric graphs.

\begin{theorem}\label{t:2ec}
Let $\mg$ be a $2$-edge-connected metric graph, and let $(G,l)$ denote its canonical loopless model.
Then \tfae:
\begin{enumerate}
 \item $\mg$ is hyperelliptic.
 \item There exists an involution $i:G\rightarrow G$ such that $G/i$ is a tree.
 \item There exists a nondegenerate harmonic morphism of degree $2$ from $G$ to a tree, or $|V(G)|=2$.
\end{enumerate}
\end{theorem}

\begin{proof} First, if $|V(G)|=2$, then $G$ must be the unique graph on 2 vertices and $n$ edges, and so all three conditions hold.  So suppose $|V(G)|>2$.  Let us prove (i) $\Rightarrow$ (ii).

Let $D\in\Div \mg$ be a divisor of degree $2$ and rank $1$. We will now define an involution $i:\mg\rightarrow\mg$ and then show that it induces an involution $i:G\rightarrow G$ on its model.  By slight abuse of notation, we will call both of these involutions $i$.
We will then prove that $G/i$ is a tree.

We define $i:\mg\rightarrow\mg$ as follows.  Given $x\in \mg$, since $D$ has rank $1$, there exists $x'\in \mg$ such that $D\sim x+x'$. Since $\mg$ is $2$-edge-connected, this $x'$ is unique.  Then let $i(x)=x'$. Clearly $i$ is an involution on sets.

\begin{claim}\label{c:iso}
The map $i:\mg\rightarrow\mg$ is an isometry.
\end{claim}
\begin{proof}[Proof of Claim \ref{c:iso}]
Let $x,y\in \mg$, and let $x'=i(x)$, $y'=i(y)$. We will show that
$$d(x,y)=d(x',y')$$
\wrt\ the shortest path metric on $\mg$. We may assume $x\neq y$ and $x\neq y'$, for otherwise the statement is clear.

By construction, we have $x+x'\sim y+y'$, so let $f:\mg\rightarrow\mathbb{R}$ be a rational function on $\mg$ such that
$$\divi f=y+y'-x-x'.$$
By Lemma~\ref{l:slopes}(i), we have $\partial M(f)=\left\{ x,x'\right\}$ and $\partial m(f)=\left\{ y,y'\right\}$.  Furthermore, $M(f)$ has precisely two edges leaving it, one at $x$ and one at $x'$.  Similarly, $m(f)$ has precisely two edges leaving it, one at $y$ and one at $y'$.

By a {\bf path down from \ensuremath{x}} we mean a path in $\mg$ that starts at $x$ and along which $f$ is decreasing, and which is as long as possible given these conditions. By Lemma~\ref{l:slopes}(ii), $f$ must have constant slope $-1$ along such a path.  Furthermore, by Lemma~\ref{l:slopes}(iii), any path down from $x$ must end at $y$ or $y'$. After all, if it ended at some point $w$ with $\divi f(w)=0$, then it would not be longest possible.  Similarly, every path down from $x'$ must end at $y$ or $y'$.

So let $P$ be a path down from $x$, and let $l$ be its length in $\mg$. Suppose that $P$ ends at $y$. Note that $d(x,y)=l$. Indeed if there were a shorter path from $x$ to $y$, then $f$ would have an average slope of less than $-1$ along it, contradicting Lemma~\ref{l:slopes}(ii). Also, by Lemma~\ref{l:slopes}(iii), there must exist a path $P'$ down from $x'$ to $m(f)$ that is edge-disjoint from $P$, so $P'$ must end at $y'$. Now, $P$ and $P'$  have the same length since $f$ decreases from its maximum to its minimum value at constant slope $-1$ along each of them. Thus $d(x',y')=l$, and $d(x,y)=d(x',y')$ as desired.

So we may assume instead that every path $P$ down from $x$ ends at $y'$. Then every path down from $x'$  must end at $y$, and no pair of $x$--$y'$ and $x$--$y$ paths has a common vertex, for otherwise we could find a path down from $x$ that ends at $y$.
Now, all of these paths have a common length, say $l$, by the argument above. Let $d_{x}$ be the length of the shortest path in $M(f)$ between $x$ and $x'$, and let $d_{y}$ be the length of the shortest path in $m(f)$ between $y$ and $y'$. Then $$d(x,y)=\min(d_{x},d_{y})+l=d(x',y'),$$
proving the claim.
\end{proof}

Thus the involution $i:\mg\rightarrow\mg$ induces an automorphism $i:G\rightarrow G$ of canonical loopless models.
Then by Lemma \ref{l:qharmonic}, it induces a quotient morphism
$$\pi:G\rightarrow G/i$$
which is a nondegenerate harmonic morphism of degree two.  In particular, we note that $G/i$ does not consist of a single vertex, since $G$ had more than two vertices by assumption.

The final task is to show that $G/i$ is a tree. In \cite{bn}, this is done by by showing that $\Jac\left(G/i\right)=1$ using the pullback of Jacobians along harmonic morphisms. Here, we will instead prove directly, in several steps, that the removal of any edge $\tilde{e}$ from $G/i$ disconnects it.  If so, then $G/i$ is necessarily a tree.

\begin{claim}\label{c:2edges} Let $\tilde{e}\in E(G/i)$.  Then $\pi^{-1}(\tilde{e})$ consists of two edges.
\end{claim}

\begin{proof}[Proof of Claim \ref{c:2edges}]
Suppose instead that $\pi^{-1}(\tilde{e})$ consists of a single edge $e=xy\in E\left(G\right)$.  Then $i$ fixes $e$, and $i(x)=x$ and $i(y)=y$. Regarding $e$ as a segment of $\mg$, choose points $x_{0},y_{0}$ in the interior of $e$ such that the subsegment $x_{0}y_{0}$ of $e$ has length $<d\left(x,y\right)$. Then $i(x_0)=x_0$ and $i(y_0)=y_0$, so $D\sim2x_{0}\sim2 y_0$.  So there exists $f:G\rightarrow\mathbb{R}$ with $\divi f=2x_{0}-2y_{0}$. Then there are two paths down from $x_{0}$ to $y_{0}$, necessarily of the same length in $G$; but one of them is the segment $x_{0}y_{0}$ and the other one passes through $x$ and $y$, contradiction.
\end{proof}

\begin{claim} \label{c:2ecut} Let $\tilde{e}\in E(G/i)$, and let $\pi^{-1}(\tilde{e}) = \{e,e'\}$.  Let $e$ have vertices $x,y$ and $e'$ have vertices $x',y'$, labeled so that $\pi(x)=\pi(x')$ and $\pi(y)=\pi(y')$.
Then in $G\setminus\{e,e'\}$, there is no path from $\{x,x'\}$ to $\{y,y'\}$.
\end{claim}

\begin{proof}[Proof of Claim \ref{c:2ecut}] By definition of the quotient map $\pi$, we have $i(x)=x'$ and $i(y)=y'$.  
Since $x+x'\sim y+y'$, we may pick a rational function $f$ on $\mg$ such that
$$\divi f=y+y'-x-x'.$$
Then $f$ is linear along both $e$ and $e'$.  By Lemma~\ref{l:slopes}(i), we have $\partial M(f)=\{ x,x'\}$ and $\partial m (f)=\{ y,y'\}$, and by Lemma~\ref{l:slopes}(ii), $f$ must have constant slope $1$ along each of $e$ and $e'$, decreasing from $\{x,x'\}$ to $\{y,y'\}$.  Then again by Lemma~\ref{l:slopes}(i), $e$ and $e'$ separate $\{x,x'\}$ from $\{y,y'\}$.
\end{proof}

\begin{claim} \label{c:allcut} Any edge $\tilde{e}\in E(G/i)$ is a cut edge, that is, its removal disconnects $G/i$.
\end{claim}

\begin{proof}[Proof of Claim \ref{c:allcut}] By Claim \ref{c:2edges}, we may let $\pi^{-1}(\tilde{e}) = \{e,e'\}$.  Let $e$ have vertices $x,y$ and $e'$ have vertices $x',y'$, and let $\tilde{x}=\pi(x)=\pi(x')$ and $\tilde{y}=\pi(y)=\pi(y')$ be the vertices of $\tilde{e}$.

Suppose for a contradiction that there is a path in $G/i$ from $\tilde{x}$ to $\tilde{y}$ not using $\tilde{e}$.  By nondegeneracy of $\pi$, we may lift that path to a path in $G$ from $x$ to $y$ or $y'$ that does not use $e$ or $e'$.  But this contradicts Claim~\ref{c:2ecut}.
\end{proof}

We conclude that $G/i$ is a tree. We have shown (i) implies (ii).

(ii) $\Rightarrow$ (iii). This is precisely Lemma \ref{l:qharmonic}.

 (iii) $\Rightarrow$(i). This argument follows \cite{bn}.  Let $\phi:(G,l)\rightarrow (T,l')$ be a nondegenerate degree 2 harmonic morphism of loopless models, where $T$ is a tree.  Let $\phi:\mg\rightarrow\mt$ be the induced map on metric graphs, also denoted $\phi$ by abuse of notation.  Pick any $y_{0}\in \mt$ and, regarding $y_{0}$ as a divisor of degree $1$, let $D=\phi^{*}(y_{0})$.  So $D$ is an effective divisor on $\mg$ of degree $2$, say $D=y+y'$. Now we have
$r(D )\leq1$, for otherwise, for any $z\in\mg$, we would have $(y+y')-(z+y')\sim 0$, so $y\sim z$, so $\mg$ itself would be a tree.  It remains to show $r(D)\geq1$. Let $x\in \mg$. Then $\phi(x)\sim y_{0}, \text{ since } \Jac T=1$. Then by Proposition \ref{p:pushpull}, we have $$D=\phi^{*}(y_{0})\sim\phi^{*}(\phi(x))$$
$$=m_{\phi}(x)\cdot x+E$$
for $E$ an effective divisor on $\mg$. Since $\phi$ is nondegenerate, we have $m_{\phi}(x)\geq1$, so $D\sim x+E'$ for some effective $E'$, as desired. This completes the proof of Theorem~\ref{t:2ec}.
\end{proof}

In the next section, we will use the fact that the hyperelliptic involution in Theorem~\ref{t:2ec}(ii) is unique.  To prove uniqueness, we need the following fact.

\begin{proposition}\label{p:5.5}\cite[Proposition 5.5]{bn}
If $D$ and $D'$ are degree 2 divisors on a metric graph $\mg$ with $r(D) = r(D') = 1$, then $D \sim D'$.
\end{proposition}

\begin{remark}
The proof of \cite[Proposition 5.5]{bn} extends immediately to metric graphs.
\end{remark}

\begin{corollary}\label{c:5.14} \cite[Corollary 5.14]{bn}
Let $\mg$ be a 2-edge-connected hyperelliptic metric graph.  Then there is a {\em unique} involution $i:\mg\rightarrow\mg$ such that $\mg/i$ is a tree.
\end{corollary}

\begin{proof}
Suppose $i$ and $i'$ are two such involutions.  From the proof of Theorem~\ref{t:2ec}, we see that for any $x\in \mg$, $x+i(x)$ and $x+i'(x)$ are both divisors of rank 1.  Then by Proposition~\ref{p:5.5}, $x+i(x)\sim x+i'(x)$, so $i(x)\sim i'(x)$.  Since $\mg$ is 2-edge-connected, it follows that $i(x)=i'(x)$.
\end{proof}

Our next goal is to prove Theorem \ref{t:main}, which extends Theorem \ref{t:2ec} to graphs with bridges. Theorem \ref{t:main} has no analogue in \cite{bn} because it takes advantage of the integer stretching factors present in morphisms of metric graphs which do not occur in combinatorial graphs.

Let us say that a point $s$ in a metric graph $\mg$ \emph{separates} $y,z \in \mg$ if every $y$--$z$ path in $\mg$ passes through $s$, or equivalently, if $y$ and $z$ lie in different connected components of $\mg \setminus s$.  A \emph{cut vertex} is a point $x\in\mg$ such that $\mg\setminus x$ has more than one connected component.

\begin{lemma}
\label{l:cut}
Let $\mg$ be a 2-edge-connected hyperelliptic metric graph with hyperelliptic involution $i$, and let $x$ be a cut vertex of $\mg$. Then $i(x)=x$.
\end{lemma}
\begin{proof}
 Suppose $i(x) \neq x$. Since $x$ is a cut vertex, we may pick some $y \in \mg$ such that $x$ separates $y$ and $i(x)$, so $i(x)$ does not separate $x$ and $y$. Furthermore, since $i$ is an automorphism, $i(x)$ separates $i(y)$ and $x$. Therefore $i(x)$ separates $y$ and $i(y)$. But  Lemma \ref{l:slopes}(i), applied to a rational function $f$ with
\[ \divi f = y + i(y) - x - i(x), \]
yields the existence of a $y$---$i(y)$ path in $m(f)$ not passing through $i(x)$, contradiction.
\end{proof}

\begin{lemma}
\label{l:bridge}
Let $\mg$ be a metric graph, and let $\mg'$ be the metric graph obtained from $\mg$ by contracting all bridges. Let $\varphi: \mg \to \mg'$ be the natural contraction morphism. Given $D \in \Div \mg$, let $D' \in \Div \mg'$ be given by
\[ D' = \sum_{x \in \mg} D(x) \cdot \varphi(x). \]
Then $D \in \Prin \mg$ if and only if $D' \in \Prin \mg'$.
\end{lemma}

\begin{proof}
By induction, we may assume that $\mg'$ was obtained by contracting a single bridge $e$ in $\mg$, and that $D$ is not supported on the interior of $e$. Let $e$ have endpoints $x_1$ and $x_2$, and for $i=1,2,$ let $\Gamma_i$ denote the connected component of $x_i$ in $\mg \setminus e$. Let $\Gamma_i' = \varphi(\Gamma_i)$, and let $x' = \varphi(x_1) = \varphi(x_2) = \varphi(e)$.

Now suppose $D \in \Prin \mg$, so let $D = \divi f$ for $f$ a rational function on $\mg$. Define $f'$ on $\mg$ as follows: let   $f'(\varphi(y)) = f(y)$ for $y \in \Gamma_1$, and $f'(\varphi(y)) = f(y)+f(x_1)-f(x_2)$ for $y \in \Gamma_2$. This uniquely defines a rational function $f'$ on $\mg'$, and one can check that $\divi f' = D'$.

Conversely, suppose $D' = \divi f' \in \Prin \mg'$. Let $$m = D(x_1) - \sum \text{outgoing slopes from } x' \text{ into } \Gamma_1'.$$ Now define $f$ on $\mg$ as follows: let $f(x) = f'(\varphi(x))$ if $x \in \Gamma_1$, let $f(x) = f'(\varphi(x)) + m \cdot l(e)$ if $x \in \Gamma_2$, and let $f$ be linear with slope $m$ along $e$. Then one can check that $f$ is a rational function on $\mg$ with $\divi f = D$.
\end{proof}

It follows from Lemma \ref{l:bridge} that the rank of divisors is preserved under contracting bridges.  The argument can be found in \cite[Corollaries 5.10, 5.11]{bn}, and we will not repeat it here. In particular, we obtain

\begin{corollary}
\label{c:bridge}
Let $\mg$ be a metric graph, and let $\mg'$ be the metric graph obtained from $\mg$ by contracting all bridges. Then $\mg$ is hyperelliptic if and only if $\mg'$ is hyperelliptic.
\end{corollary}

We can finally prove the main theorem of the section, which generalizes Theorem \ref{t:2ec}.
\begin{theorem}
\label{t:main}
Let $\mg$ be a metric graph with no points of valence 1, and let $(G,l)$ denote its canonical loopless model. Then the following are equivalent:
\begin{enumerate}
\item $\mg$ is hyperelliptic.
\item There exists an involution $i: G \to G$ such that $G/i$ is a tree.
\item There exists a nondegenerate harmonic morphism of degree 2 from $G$ to a tree, or $|V(G)|=2$.
\end{enumerate}
\end{theorem}

\begin{proof}
As in the proof of Theorem \ref{t:2ec}, we may assume that $|V(G)|>2$. In fact, the proofs of (ii) $\Rightarrow$ (iii) and (iii) $\Rightarrow$ (i) from Theorem \ref{t:2ec} still hold here, since they do not rely on 2-edge-connectivity. We need only show (i) $\Rightarrow$ (ii).

Let $\mg$ be a hyperelliptic metric graph with no points of valence 1, and let $\mg'$ be obtained by contracting all bridges of $\mg$. Then $\mg'$ is hyperelliptic by Corollary \ref{c:bridge}. Since $\mg$ had no points of valence 1, the image of any bridge in $\mg$ is a cut vertex $x \in \mg'$. By Lemma \ref{l:cut}, all such vertices are fixed by the hyperelliptic involution $i'$ on $\mg'$. Thus, we can extend $i'$ uniquely to an involution $i$ on $\mg$ by fixing, pointwise, each bridge of $\mg$. Then $i$ is also an involution on the canonical loopless model $(G,l)$, and $G/i$ is a tree whose contraction by the images of the bridges of $G$ is the tree~$\mg'/i'$.
\end{proof}

\begin{remark}
The requirement that $\mg$ has no points of valence 1 is not important, because Corollary~\ref{c:bridge} allows us to contract such points away.  Also, because of Definition~\ref{d:thyp} and the stability condition on tropical curves, we will actually never encounter points of valence 1 in the tropical context.
\end{remark}

\section{The hyperelliptic locus in tropical $M_g$}
Which tropical curves of genus $g$ are hyperelliptic? In this section, we will use the main combinatorial tool we have developed, Theorem \ref{t:main}, to construct the hyperelliptic locus $H^\text{tr}_g$ in the moduli space $\Mtrg$ of tropical curves. The space $\Mtrg$ was defined in \cite{bmv} and computed explicitly for $g \leq 5$ in \cite{ch}. It is $(3g-3)$-dimensional and has the structure of a stacky fan, as defined in \cite[Definition 3.2]{ch}.

It is a well-known fact that the classical hyperelliptic locus $\mathcal{H}_g \subset \mathcal{M}_g$ has dimension $2g-1$. Therefore, it is surprising that $H^{\tr}_g$ is actually $(3g-3)$-dimensional, as observed in \cite{lpp}, especially given that tropicalization is a dimension-preserving operation in many important cases \cite{bg}. However, if one considers only 2-edge-connected tropical curves, then the resulting locus, denoted $H^{(2),\tr}_g$, is in fact $(2g-1)$-dimensional. The combinatorics of $H_g^{(2),\tr}$ is nice, too: in Theorem \ref{t:ladder} we prove that the $(2g-1)$-dimensional cells are graphs that we call ladders of genus $g$. See Definition~\ref{d:ladder} and Figure \ref{f:ladder}. We then explicitly compute the spaces $H_3^{\tr}$, $H_3^{(2),\tr}$ and~$H_4^{(2),\tr}$.

\subsection{Construction of $H_g^{\tr}$ and $H_g^{(2),\tr}$}

We will start by giving a general framework for constructing parameter spaces of tropical curves in which edges may be required to have the same length. We will see that the loci $H^{\tr}_g$ and $H^{(2),\tr}_g$ of hyperelliptic and 2-edge-connected hyperelliptic tropical curves, respectively, fit into this framework.

Recall that a {\bf combinatorial type} of a tropical curve is a pair $(G,w)$, where $G$ is a connected, non-metric multigraph, possibly with loops, and $w:V(G) \to \Z_{\geq 0}$ satisfies the stability condition that if $w(v) = 0$ then $v$ has valence at least 3. The {\bf genus} of $(G,w)$ is
\[ |E(G)|-|V(G)|+1+\sum_{v \in V(G)} w(v). \]

Now, a {\bf constrained type} is a triple $(G,w,r)$, where $(G,w)$ is a combinatorial type and $r$ is an equivalence relation on the edges of $G$. We regard $r$ as imposing the constraint that edges in the same equivalence class must have the same length.  Given a constrained type $(G,w,r)$ and a union of equivalence classes $S=\{e_1,\dots,e_k\}$ of $r$, define the contraction along $S$ as the constrained type $(G',w',r')$. Here, $(G',w')$ is the combinatorial type obtained by contracting all edges in $S$.
Contracting a loop, say at vertex $v$, means deleting it and adding $1$ to $w(v)$. Contracting a nonloop edge, say with endpoints $v_1$ and $v_2$, means deleting that edge and identifying $v_1$ and $v_2$ to obtain a new vertex whose weight is $w(v_1)+w(v_2)$. Finally, $r'$ is the restriction of $r$ to $E(G)\setminus S$.

An automorphism of $(G,w,r)$ is an automorphism $\varphi$ of the graph $G$ which is compatible with both $w$ and $r$. Thus, for every vertex $v \in V(G)$, we have $w(\varphi(v))=w(v)$, and for every pair of edges $e_1,e_2 \in E(G)$, we have that $e_1 \sim_r e_2$ if and only if $\varphi(e_1) \sim_r \varphi(e_2)$. Let $N_r$ denote the set of equivalence classes of $r$. Note that the group of automorphisms $\Aut(G,w,r)$ acts naturally on the set $N_r$, and hence on the orthant $\R^{N_r}_{\geq 0}$, with the latter action given by permuting coordinates. We define $\overline{C(G,w,r)}$ to be the topological quotient space
\[ \overline{C(G,w,r)} = \frac{\R^{N_r}_{\geq 0}}{\Aut(G,w,r)}. \]

Now, suppose $\mathcal{C}$ is a collection of constrained types that is closed under contraction. Define an equivalence relation $\sim$ on the points in the union
\[ \coprod_{(G,w,r)\in \mathcal{C}} \overline{C(G,w,r)} \]
as follows. Regard a point $x \in \overline{C(G,w,r)}$ as an assignment of lengths to the edges of $G$ such that $r$-equivalent edges have the same length. Given two points $x \in \overline{C(G,w,r)}$ and $x'\in\overline{C(G',w',r')}$, let $x \sim x'$ if the two tropical curves obtained by contracting all edges of length zero are isomorphic.

Now define the topological space $M_\mathcal{C}$ as
\[ M_\mathcal{C} = \coprod \overline{C(G,w,r)}/\sim, \]
where the disjoint union ranges over all types $(G,w,r)\in\mathcal{C}$.
Since $\mathcal{C}$ is closed under contraction, it follows that the points of $M_\mathcal{C}$ are in bijection with $r$-compatible assignments of positive lengths to $E(G)$ for some $(G,w,r) \in \mathcal{C}$.

\begin{theorem}
  \label{t:mc}
  Let $\mathcal{C}$ be a collection of constrained types, as defined above, that is closed under contraction.
  Then the space $M_\mathcal{C}$ is a stacky fan, with cells corresponding to types $(G,w,r)$ in $\mathcal{C}$.
\end{theorem}

\begin{proof}
  The proof of Theorem \ref{t:mc} is entirely analogous to the proof that the moduli space of genus $g$ tropical curves is a stacky fan, so we refer the reader to \cite[Theorem 3.4]{ch}.
\end{proof}

\begin{remark}
Note that the stacky fan $M_\mathcal{C}$ is not in general a stacky subfan of $\Mtrg$.  Instead, it may include only parts of the cells of $\Mtrg$, since edges may be required to have equal length.
\end{remark}

Our next goal is to define the collections $\mathcal{C}_g$ and $\mathcal{C}_g^2$ of hyperelliptic and 2-edge-connected hyperelliptic types of genus $g$.  If $(G,w)$ is a combinatorial type, let $G^w$ denote the graph obtained from $G$ by adding $w(v)$ loops at each vertex $v$.  Let $G^w_-$ be the loopless graph obtained by adding a vertex to the interior of each loop in $G^w$.  Now let $G$ and $G'$ be loopless graphs.  Then a morphism $\phi:G\rightarrow G'$ is
 a map of sets
$V\left(G\right)\cup E\left(G\right){\rightarrow}V\left(G'\right)\cup E\left(G'\right)$
such that
\begin{enumerate}
	\item $\phi(V(G))\subseteq V(G')$,
	\item if $e=xy$ is an edge of $G$ and $\phi\left(e\right)\in V\left(G'\right)$ then $\phi\left(x\right)=\phi\left(e\right)=\phi\left(y\right)$, and
	\item if $e=xy$ is an edge of $G$ and $\phi(e)\in E\left(G'\right)$ then $\phi(e)$ is an edge between $\phi(x)$ and $\phi(y)$.
\end{enumerate}
The morphism $\phi$ is harmonic if the map $\bar{\phi}:(G,{\bf 1})\rightarrow (G',{\bf 1})$ of loopless models is a harmonic morphism, where ${\bf 1}$ denotes the function assigning length 1 to every edge.  The definitions above follow \cite{bn}.

\begin{definition}
A constrained type $(G,w,r)$ is {\bf 2-edge-connected hyperelliptic} if

\begin{enumerate}
	\item $G$ is 2-edge-connected,
	\item the loopless graph $G^w_-$ has a nondegenerate harmonic morphism $\phi$ of degree 2 to a tree, or $|V(G)|=2$, and
	\item the relation $r$ is induced by the fibers of $\phi$ on nonloop of $G$, and is trivial on the loops of $G$.
\end{enumerate}  
The type $(G,w,r)$ is said to be {\bf hyperelliptic} if $r$ is the trivial relation on bridges and the type  $(G',w',r')$ obtained by contracting all bridges is  2-edge-connected hyperelliptic in the sense we have just defined.
\end{definition}

Let $\mathcal{C}_g$ denote the collection of hyperelliptic types of genus $g$, and $\mathcal{C}_g^{(2)}$ the collection of 2-hyperelliptic types of genus $g$.

\begin{proposition}
The collections $\mathcal{C}_g$ and $\mathcal{C}_g^2$ of hyperelliptic and 2-edge-connected hyperelliptic types defined above are closed under contraction.
\end{proposition}
\begin{proof}
We will check that $\mathcal{C}_g^{(2)}$ is closed under contraction. If so, then $\mathcal{C}_g$ is too, since it was defined precisely according to contractions to types in $\mathcal{C}_g^{(2)}$. Note also that by definition, contraction preserves the genus $g$.

Let $(G,w,r)$ be a 2-hyperelliptic type, and let
\[ \varphi: G^w_- \rightarrow T \]
be the unique harmonic morphism to a tree $T$ that is either nondegenerate of degree 2, or $T$ has 1 vertex and $G^w_-$ has 2 vertices. Let $S$ be a class of the relation $r$. Then there are three cases: either $S=\{e\}$ and $e$ is a loop; or $S = \{e\}$, $e$ is a nonloop edge, and $\varphi(e) \in V(T)$; or $S=\{e_1,e_2\}$ and $\varphi(e_1) = \varphi(e_2)$ is an edge of $T$.

Suppose $S = \{e\}$ and $e$ is a loop. Then the contraction of $(G,w,r)$ by $e$ is $(G/e,w',r|_{E(G/e)})$, where $w'$ is obtained from $w$ by adding 1 at $v$. Then
\[ (G/e)^{w'}_- = G^w_- \]
so the same morphism $\varphi$ shows that $(G/e,w',r|_{E(G/e)})$ is a 2-hyperelliptic type.

Suppose $S=\{e\}$, $e$ is a nonloop edge, and, regarding $e$ as an edge of $G^w_-$, we have $\varphi(e) = v \in V(T)$. Let $k$ be the number of edges in $G$ that are parallel to $e$, and let $T'$ be the tree obtained from $T$ by adding $k$ leaves at $v$. Then we may construct a morphism
\[ \varphi' : (G/e)^{w'}_- \rightarrow T' \]
which in particular sends the $k$ edges parallel to $e$ to the $k$ new leaf edges of $T'$, and which has the required properties. See Figure \ref{f:case2}.

Suppose $S=\{e_1,e_2\}$ where $e_1$ and $e_2$ are nonloop edges of $G$ and $\varphi(e_1)=\varphi(e_2)=e \in E(T)$. Now, contracting $\{e_1,e_2\}$ in $G^w_-$ may create new loops, say $k$ of them, as illustrated in Figure \ref{f:case3}. Let $T'$ be the tree obtained from $T$ by adding $k$ leaves at either end of the edge $e$ and then contracting $e$. Then we may construct a harmonic morphism
\[ \varphi': (G/\{e_1,e_2\})^w_- \rightarrow T' \]
which sends the $k$ new loops to the $k$ new leaves of $T'$, and which has the required properties.
\end{proof}

\begin{figure}[h!]
\includegraphics[width=3in]{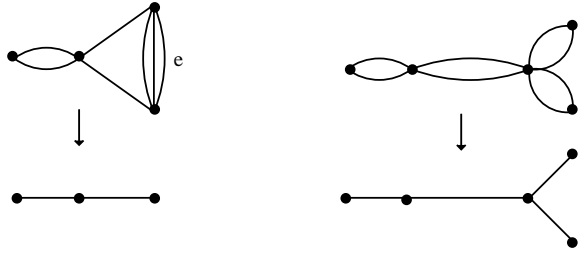}%
\caption{Contracting a vertical edge.}
\label{f:case2}
\end{figure}

\begin{figure}[h!]
\includegraphics[width=3in]{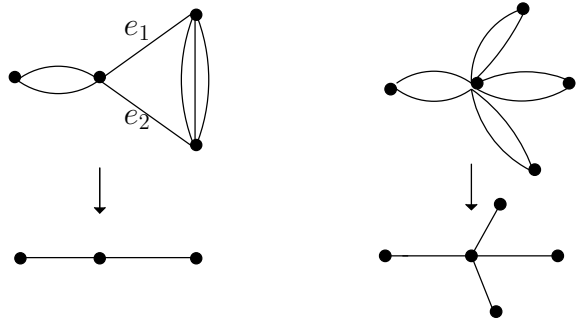}%
\caption{Contracting two horizontal edges.}
\label{f:case3}
\end{figure}

\begin{definition}
  The space $H_g^{\tr}$ of tropical hyperelliptic curves of genus $g$ is defined to be the space $M_{\mathcal{C}_g}$, where $\mathcal{C}_g$ is the collection of hyperelliptic combinatorial types of genus $g$ defined above.

The space $H^{(2),\tr}_g$ of 2-edge-connected tropical hyperelliptic curves of genus $g$ is defined to be the space $M_{\mathcal{C}^{(2)}_g}$, where $\mathcal{C}^{(2)}_g$ is the collection of 2-edge-connected hyperelliptic combinatorial types of genus $g$ defined above.
\end{definition}

\begin{proposition} $ $
\begin{enumerate}
	\item The points of $H_g^{(2),\tr}$ are in bijection with the 2-edge-connected hyperelliptic tropical curves of genus $g$.
\item The points of $H_g^{\tr}$ are in bijection with the hyperelliptic tropical curves of genus $g$.

\end{enumerate}
\end{proposition}

\begin{proof}
Regard a point in $H_g^{(2),\tr}$ as an assignment $l':E(G)\rightarrow \R$ of positive lengths to the edges of $G$, where $(G,w,r)$ is some 2-edge-connected hyperelliptic type.  Then there is a degree 2 harmonic morphism of loopless graphs $G^w_-\rightarrow T$ inducing the relation $r$ on $E(G)$, or $|V(G)|=2$ in which case $(G,l)$ is clearly hyperelliptic.  Then there is a a degree 2 harmonic morphism of loopless models $\phi:(G^w_-, l)\rightarrow (T,l'')$.  Here, $l$ agrees with $l'$ on nonloop edges of $G$, is uniformly 1, say, on the $2 w(v)$ added half-loops of $G^w_-$ at each vertex, and is $l'(e)/2$ on the half-loops of $G^w_-$ corresponding to each loop $e\in E(G)$.  Furthermore, $l''$ is defined in the natural way, so that the harmonic morphism $\phi$ uses only the stretching factor 1.  

Then by Theorem~\ref{t:2ec}, the loopless model we constructed is 2-edge-connected hyperelliptic. Reversing the construction above shows that the map from $H_g^{(2),\tr}$ to the set of 2-edge-connected tropical hyperelliptic curves is surjective.  On the other hand, it is injective by Corollary~\ref{c:5.14}.
Finally, part (ii) follows from part (i), using Corollary~\ref{c:bridge} and the fact that the  hyperelliptic types are constructed by adding bridges to 2-edge-connected hyperelliptic~types.
\end{proof}

\subsection{Maximal cells of $H_g^{(2),\tr}$}
Now we will prove that $H_g^{(2),\tr}$ is pure of dimension $2g-1$, and we will characterize its maximal cells. Recall that the dimension of a cell of the form $\R^N_{\geq 0} / G$ is equal to $N$, and the dimension of a stacky fan is the largest dimension of one of its cells. It is {\bf pure} if all of its maximal cells have the same dimension.

First, we define the graphs which, as it turns out, correspond to the maximal cells of $H_g^{(2),\tr}$.

\begin{definition}\label{d:ladder}
Let $T$ be any nontrivial tree with maximum valence $\leq 3$. Construct a graph $L(T)$ as follows. Take two disjoint copies of $T$, say with vertex sets $\{v_1,\dots,v_n\}$ and $\{v_1',\dots,v_n'\}$, ordered such that $v_i$ and $v_i'$ correspond. Now, for each $i=1,\dots,n$, consider the degree $d= \deg v_i = \deg v_i'$. If $d=1$, add two edges between $v_i$ and $v_i'$. If $d =2$, add one edge between $v_i$ and $v_i'$. The resulting graph $L(T)$ is a loopless connected graph, and by construction, every vertex has valence 3. We call a graph of the form $L(T)$ for some tree $T$ a {\bf ladder}. See Figure \ref{f:ladder}.
\end{definition}

\begin{figure}[h!]
\includegraphics[height=.6in]{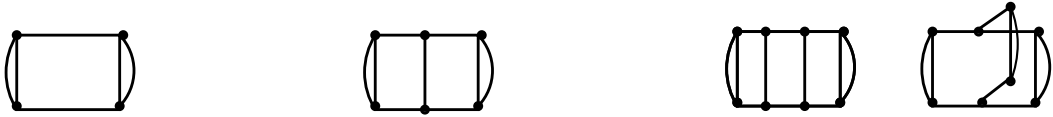}%
  \caption{The ladders of genus 3, 4, and 5.}
  \label{f:ladder}
\end{figure}

\begin{lemma}\label{l:laddergenus}
Let $T$ be a tree on $n$ vertices with maximum degree at most 3. Then the genus of the graph $L(T)$ is $n+1$.
\end{lemma}

\begin{proof}
For $i=1,2,3$, let $n_i$ denote the number of vertices in $T$ of degree $i$. Then $L(T)$ has $2n$ vertices and $2(n-1)+2n_1+n_2$ edges; hence its genus is $g=2n_1+n_2-1$. Also, double-counting vertex-edge incidences in $T$ gives $2(n-1)=n_1+2n_2+3n_3$. Adding, we have
\[ g+2(n-1)=3n_1+3n_2+3n_3-1 = 3n-1, \]
so $g=n+1$.
\end{proof}

\begin{theorem}
\label{t:ladder}
Fix $g \geq 3$. The space $H_g^{(2),\tr}$ of 2-edge-connected hyperelliptic tropical curves of genus $g$ is a stacky fan which is pure of dimension $2g-1$. The maximal cells correspond to ladders of genus $g$.
\end{theorem}

\begin{remark}
Note that the stacky fan  $H_g^{(2),\tr}$ is naturally a closed subset of $\Mtrg$, but it is not a stacky subfan because sometimes edges are required to have equal lengths.  So its stacky fan structure is more refined than that of $\Mtrg$.  Compare Figure \ref{f:h23} and \cite[Figure 1]{ch}.
\end{remark}

\begin{proof}
Let $(G,w,r) \in \mathcal{C}_g^{(2)}$ be the type of a maximal cell in $H_g^{(2),\tr}$. Our goal is to show that $w \equiv 0$, that $G=L(T)$ for some tree $T$, and that $r$ is the relation on $E(G)$ induced by the natural harmonic morphism
\[ \varphi_T: L(T) \rightarrow T \]
of degree 2.

First, we observe that the dimension of the cell $\overline{C(G,w,r)}$ is, by construction, the number of equivalence classes of $r$. Now, we have, by definition of $\mathcal{C}_g^{(2)}$, a morphism
\[ \varphi: G^w_- \rightarrow T \]
where either $G^w_-$ has 2 vertices and $T$ is trivial, or $T$ is nontrivial and $\varphi$ is a nondegenerate harmonic morphism of degree 2.

We immediately see that $w$ is uniformly zero, for otherwise the cell $\overline{C(G^w,\underline{0},r')}$ would contain $\overline{C(G,w,r)}$, contradicting maximality of the latter. Here, $r'$ denotes the relation on $E(G^w)$ that is $r$ on $E(G)$ and trivial on the added loops of $G^w$.

Let us also dispense with the special case that $G_-$ has 2 vertices: if so, then $G$ is the unique graph consisting of $g+1$ parallel edges. But then our cell $\overline{C(G,\underline{0},r)}$ is far from maximal; in fact, it is contained in each cell corresponding to a ladder.

Therefore we may assume that $T$ is nontrivial and $\varphi: G_- \to T$ is a nondegenerate degree 2 harmonic morphism.

Next, we claim that every vertex $v \in V(G_-)$ has horizontal multiplicity $m_\varphi(v)=1$. This claim shows in particular that $G$ has no loops. In fact, the intuition behind the claim is simple: if $m_\varphi(v)=2$, then split $v$ into two adjacent vertices $v'$ and $v''$ with $\varphi(v')=\varphi(v'')$ to make a larger cell. However, because $G$ might a priori contain loops and thus $G_-$ might have vertices of degree 2, we will need to prove the claim carefully.  The proof is illustrated in Figure~\ref{f:split}.

To prove the claim, suppose $v \in V(G_-)$ is such that $m_\varphi(v)=2$. Let us assume that $\deg(v)>2$, i.e. that $v$ is not the midpoint of some loop $l$ of $G$, for if it is, then pick the basepoint of $l$ instead. Let $w = \varphi(v)$. Let $e_1,\dots,e_k \in E(T)$ be the edges of $T$ incident to $w$, and let $w_1,\dots,w_k$ denote their respective endpoints that are different from $w$. Now, for each $i=1,\dots,k$, the set $\varphi^{-1}(e_i)$ consists of two edges of $G_-$; call them $a_i$ and $b_i$. By renumbering them, we may assume that the first $j$ of the pairs form loops in $G$. That is, if $1 \leq i \leq j$, then the edges $a_i$ and $b_i$ have common endpoints $v_i$ and $v$, with $\deg(v_i)=2$.

Let us construct a new graph $G'$ and relation $r'$ such that $(G,\underline{0},r)$ is a contraction of $(G',\underline{0},r')$. Replace the vertex $v$ with two vertices, $v_a$ and $v_b$, where $v_a$ is incident to edges $a_1,\dots,a_k$ and $v_b$ is incident to edges $b_1,\dots,b_k$. Now suppress the degree 2 vertices $v_1,\dots,v_j$, so that for $1 \leq i \leq j$, the pair of edges $\{a_i,b_i\}$ becomes a single edge $e_i$. Finally, add an edge $e$ between $v_a$ and $v_b$. Call the resulting graph $G'$. See Figure \ref{f:split}.

\begin{figure}[h!]
\includegraphics[height=1.5in]{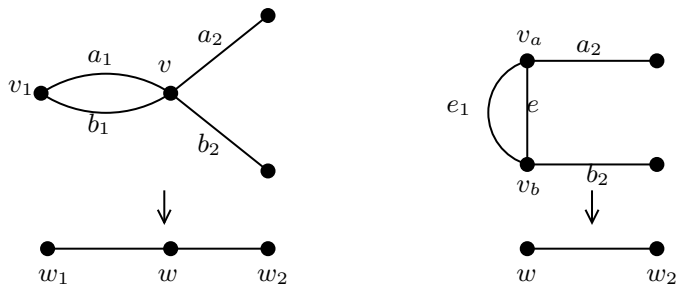}%
\caption{Splitting a vertex $v$ of horizontal multiplicity 2 in the proof of Theorem \ref{t:ladder}.}
\label{f:split}
\end{figure}

By construction, the graph $G'$ is a stable, 2-edge-connected graph with $G'/e=G$ and with the same genus as $G$. Let $T'$ be the tree obtained from $T$ by deleting $e_1,\dots,e_j$. Then there is a natural degree 2 harmonic morphism
\[ \varphi': G'_- \to T' \]
as shown in Figure \ref{f:split}, inducing a relation $r'$ on $E(G')$. Finally, note that $\overline{C(G',\underline{0},r')} \supseteq \overline{C(G,\underline{0},r)}$ is a larger cell in $H_g^{(2),\tr}$, contradiction. We have proved the claim that every $v \in V(G_-)$ satisfies $m_\varphi(v)=1$. As a consequence, every vertex of $T$ has precisely two preimages under $\varphi$. Hence $G=G_-$ has no loops.

Next, we claim that $G$ consists of two disjoint copies of $T$, say $T_1$ and $T_2$, that are sent by $\varphi$ isomorphically to $T$, plus some vertical edges. The intuition is clear: the harmonicity of $\varphi$ implies that the horizontal edges of $G$ form a twofold cover of $T$, which must be two copies of $T$ since $T$ is contractible. More formally, pick any $v \in V(T)$ and let $\varphi^{-1}(v) = \{v_1,v_2\}$. For $i=1,2$, let $T_i$ be the union of all paths from $v_i$ using only horizontal edges of $G$. Since no horizontal edges leave $T_i$, we have $\varphi(T_i)=T$, and since $m_\varphi(x)=1$ for all $x \in V(G)$, it follows that $T_1$ and $T_2$ are disjoint and that
\[ \varphi|_{T_i}: T \rightarrow T_i \]
is harmonic of degree 1 and nondegenerate, hence an isomorphism. Finally, since each edge $e \in E(T)$ has only two preimages, $T_1$ and $T_2$ account for all of the horizontal edges in $G$, and hence only vertical edges are left.

Where can the vertical edges of $G$ be? Let $v \in V(T)$ be any vertex. If $v$ has degree 1, then $G$ must have at least two vertical edges above $v$, and if $v$ has degree 2, then $G$ must have at least one vertical edge above $v$, for otherwise $G$ would not be stable. In fact, we claim that $G$ cannot have any vertical edges other than the aforementioned ones. Otherwise, $G$ would have vertices of degree at least 4, and splitting these vertices horizontally produces a larger graph $G'$, the cell of which contains $\overline{C(G,w,r)}$. See Figure \ref{f:horizontal}.

\begin{figure}[h!]
\includegraphics[height=1.5in]{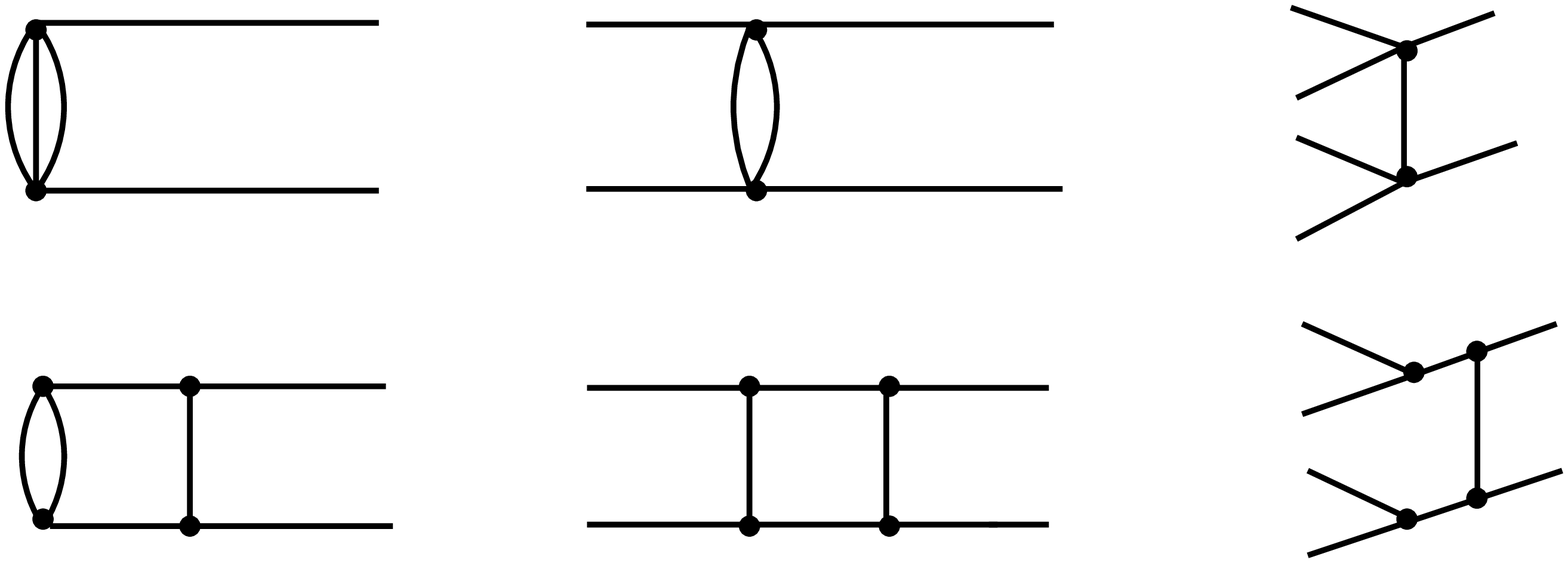}%
  \caption{Horizontal splits in $G$ above vertices in $T$ of degrees 1, 2, and 3, respectively.}
  \label{f:horizontal}
\end{figure}

For the same reason, $T$ cannot have vertices of degree $\geq 4$. We have shown that every maximal cell of $H_g^{(2),\tr}$ is a ladder of genus $g$. If $L(T)$ is a ladder of genus $g$ with tree $T$, then $T$ has $g-1$ nodes by Lemma~\ref{l:laddergenus}, and the dimension of the cell of $L(T)$ is $(g-2)+2n_1+n_2 = (g-2)+(g+1)=2g-1$. So all genus $g$ ladders yield equidimensional cells and none contains another. Therefore the genus $g$ ladders are precisely the maximal cells of $H_g^{(2),\tr}$, and each cell has dimension $2g-1$.
\end{proof}

\begin{corollary}
The maximal cells of $H_g^{(2),\tr}$ are in bijection with the trees on $g-1$ vertices with maximum degree 3.
\end{corollary}
\begin{proof}
Each ladder $L(T)$ of genus $g$ corresponds to the tree $T$.
\end{proof}

\begin{corollary}
Let $g \geq 3$. The number of maximal cells of $H_g^{(2),\tr}$ is equal to the $(g-2)^\text{nd}$ term of the sequence
\[ 1, 1, 2, 2, 4, 6, 11, 18, 37, 66, 135, 265, 552, 1132, 2410, 5098, \dots \]
\end{corollary}
\begin{proof}
This is sequence A000672 in \cite{oeis}, which counts the number of trees of maximum degree 3 on a given number of vertices.  Cayley obtained the first twelve terms of this sequence in 1875 \cite{cayleytrees}; see \cite{rainssloane} for a generating function.  See Figure \ref{f:ladder} for the ladders of genus 3,4, and 5, corresponding to the first three terms 1, 1, and 2 of the sequence.
\end{proof}

\begin{remark}
As pointed out in \cite{lpp}, the stacky fan $H_g^{\tr}$ is full-dimensional for each $g$. Indeed, take any 3-valent tree with $g$ leaves, and attach a loop at each leaf. The resulting genus $g$ graph indexes a cell of dimension $3g-3$. Thus it is more natural to consider the sublocus $H_g^{(2),\tr}$, at least from the point of view of dimension.
\end{remark}

One can use Theorem \ref{t:ladder} to compute $H_g^{(2),\tr}$ and $H_g^{\tr}$, the latter by adding all possible bridges to the cells of $H_g^{(2),\tr}$. We will explicitly compute $H_3^{(2),\tr}$, $H_3^{\tr}$, and $H_4^{(2),\tr}$ next. The space $H_4^{\tr}$ is too large to compute by hand.

\subsection{Computations}
	We now apply Theorem \ref{t:ladder} to explicitly compute the spaces $H_3^{(2),\tr}$, $H_3^{\tr}$, and $H_4^{(2),\tr}$.

\begin{theorem}
\label{t:h23}
The moduli space $H_3^{(2),\tr}$ of 2-edge-connected tropical hyperelliptic curves has 11 cells and $f$-vector
\[ (1,2,2,3,2,1). \]
Its poset of cells is shown in Figure \ref{f:h23}.
\end{theorem}
\begin{proof}
The space $M_3^{\tr}$ was computed explicitly in \cite[Theorem 2.13]{ch}. A poset of genus 3 combinatorial types is shown in \cite[Figure 1]{ch}. Now use Theorem \ref{t:ladder}.
\end{proof}

\begin{theorem}
\label{t:h3}
The moduli space $H_3^{\tr}$ of tropical typerelliptic curves has 36 cells and $f$-vector
\[ (1,3,6,11,9,5,1). \]
Its poset of cells is shown in Figure \ref{f:h3}.
\end{theorem}

\begin{proof}
  Apply Theorem \ref{t:h23} and Lemma \ref{l:bridge}.
\end{proof}

\begin{theorem}
\label{t:h24}
The moduli space $H_4^{(2),\tr}$ of 2-edge-connected tropical hyperelliptic curves of genus 4 has 31 cells and $f$-vector
\[ (1,2,5,6,7,6,3,1) \].
\end{theorem}

\begin{proof}
The space $M_4^{\tr}$ was computed explicitly in \cite[Theorem 2.13]{ch}, so we can apply Theorem \ref{t:ladder} directly to it.
\end{proof}

\section{Berkovich skeletons and tropical plane curves}
\label{s:skeletons}

In this section, we are interested in hyperelliptic curves in the plane over a complete, nonarchimedean valuated field $K$.  Every such curve $X$ is given by a polynomial of the form $$P=y^2+f(x)y+h(x)$$
for $f,h\in K[x]$.  We suppose that the Newton polygon of $P$ is the lattice triangle in $\R^2$ with vertices $(0,0), (2g+2,0),$ and $(0,2)$.  We write $\Delta_{2g+2,2}$ for this triangle. Since  $\Delta_{2g+2,2}$ has $g$ interior lattice points, the curve $X$ has genus $g$.
 The main theorem in this section, Theorem \ref{t:skeleton}, says that under certain combinatorial conditions, the Berkovich skeleton of $X$ is a ladder over the path  $P_{g-1}$  on $g-1$ vertices. The main tool is \cite[Corollary 6.27]{bpr}, which states that under nice conditions, an embedded tropical plane curve, equipped with a lattice length metric, faithfully represents the Berkovich skeleton of $X$.

 Theorem \ref{t:skeleton} is a first step in studying the behavior of hyperelliptic curves under the map
\[ \trop: M_g(K) \to \Mtrg,  \]
which sends a genus $g$ curve $X$ over $K$ to its Berkovich skeleton, which canonically has the structure of a metric graph with nonnegative integer weights \cite[Remark 5.51]{bpr}.  We conjecture that the locus of hyperelliptic algebraic curves of genus $g$ maps surjectively onto the tropical hyperelliptic locus.  It would be very interesting to study this tropicalization map further.

Recall that an embedded tropical curve $T \subseteq \R^2$ can be regarded as a metric space with respect to {\bf lattice length}. Indeed, if $e$ is a 1-dimensional segment in $T$, then it has a rational slope. Then a ray from the origin with the same slope meets some first lattice point $(p,q)\in \Z^2$. Then the length of $e$ in the metric space $T$ is its Euclidean length divided by $\sqrt{p^2+q^2}$.
By a {\bf standard ladder} of genus $g$, we mean the graph $L(T)$ defined above for $T$ a path on $g-1$ vertices. We will denote this graph $L_g$. Figure \ref{f:ladder} shows the standard ladders of genus 3, 4, and 5, as well as a nonstandard ladder of genus~5.

In Theorem \ref{t:skeleton}, we will consider unimodular triangulations of $\Delta_{2g+2,2}$ with bridgeless dual graph.  Such triangulations contain a maximally triangulated trapezoid of height 1 in their bottom half (see Figure \ref{f:newton}), so we study height 1 trapezoids in the next lemma.
\begin{lemma}
\label{l:trapezoid}
Fix $a,b,c,d \in \Z$ with $a<b$ and $c<d$, and let $Q$ be the lattice polytope with vertices $(a,0), (b,0), (c,1)$, and $(d,1)$. Then any unimodular triangulation of $Q$ has $N=b-a+d-c+1$ nonhorizontal edges $w_1,\dots,w_N$ satisfying $w_1=\{(a,0), (c,1)\}$, $w_N = \{(b,0), (d,1)\}$, and for each $i=1,\dots,N-1$, if $w_i$ has the form $\{(x,0),(y,1)\}$ for some $x,y \in \Z$, then
\[ w_{i+1} = \{(x+1,0), (y,1)\} \text{ or } w_{i+1} = \{(x,0), (y+1,1)\}. \]
Thus there are $\binom{b-a+d-c}{b-a}$ such triangulations.
\end{lemma}

\begin{proof}
The edge $\{(a,0), (c,1)\}$ is contained in some unimodular triangle $\Delta$, so either $\{(a+1,0), (c,1)\}$ or $\{(a,0), (c+1,1)\}$ must be present. Now induct, replacing $Q$ by $Q-\Delta$.
\end{proof}

\begin{definition}
Let $T \subseteq \R^n$ be an embedded tropical curve with $\dim H_1(X,\R) > 0$. We define the {\bf core} of $T$ to be the smallest subspace $Y \subseteq T$ such that there exists a deformation retract from $T$ to $Y$. See Figure \ref{f:newton}.
\end{definition}

\begin{figure}[h!]
\includegraphics[height=3in]{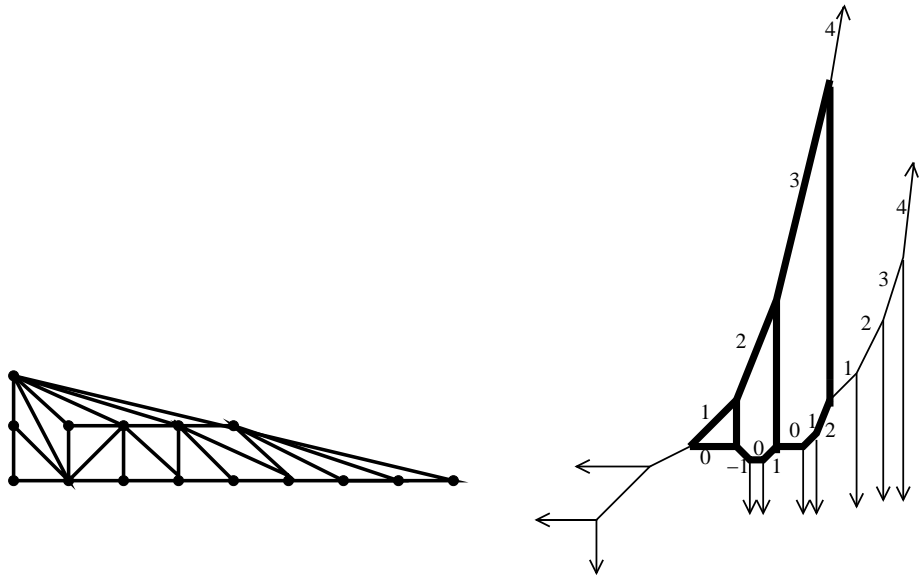}%
  \caption{On the left, a unimodular triangulation $N$ of $\Delta_{2g+2,2}$ that happens to include the edge $\{(0,2),(1,0)\}$ but not the edge $\{(0,2),(2g+1,0)\}$.  On the right, the tropical plane curve dual to $N$  has a core (shown in thick edges) that is a standard genus 3 ladder.  The numbers on the edges indicate their slopes.  We have chosen $g=3$ in this illustration.}
  \label{f:newton}
\end{figure}

\begin{theorem}
\label{t:skeleton}
Let $X \subseteq \mathbb{T}^2$ be a plane hyperelliptic curve of genus $g \geq 3$ over a complete nonarchimedean valuated field $K$, defined by a polynomial of the form $P=y^2+f(x)y+h(x)$.  Let $\widehat{X}$ be its smooth completion.  Suppose the  Newton complex of $P$ is a unimodular triangulation of the lattice triangle $\Delta_{2g+2,2}$, and suppose that the core of $\Trop X$ is bridgeless. Then the skeleton $\Sigma$ of the Berkovich analytification $\widehat{X}^{\operatorname{an}}$ is a standard ladder of genus $g$ whose opposite sides have equal length.
\end{theorem}

\begin{proof}
Let $\Sigma'$ denote the core of $T = \Trop X$, and let $N$ denote the Newton complex of $f$, to which $T$ is dual. Note that the lattice triangle $\Delta_{2g+2,2}$ has precisely $g$ interior lattice points $(1,1),\dots,(g,1)$. Now, we claim that $N$ cannot have any edge from $(0,2)$ to any point on the $x$-axis such that this edge partitions the interior lattice points nontrivially. (For example, the edge $\{(0,2),(3,0)\}$ cannot be in $N$). Indeed, if such an edge were present in $N$, then the edge in $T$ dual to it would be a bridge in $\Sigma'$, contradiction. Then since all the triangles incident to the interior lattice points are unimodular, it follows that each edge
\[ s_1 = \{(1,1),(2,1)\}, \dots, s_{g-1} = \{(g-1,1),(g,1)\} \]
is in $N$, and for the same reason, that each edge
\[ \{(0,2),(1,1)\}, \dots, \{(0,2),(g,1)\} \]
is in $N$. For $i=1,\dots,g-1$, the 2-face of $N$ above (respectively, below) $s_i$ corresponds to a vertex of $T$; call this vertex $V_i$ (respectively $W_i$).

Now, either the edge $\{(0,2),(1,0)\}$ is in $N$, or it is not. If it is, then unimodularity implies that the edges $\{(0,1),(1,0)\}$ and $\{(1,1),(1,0)\}$ are also in $N$. If not, then unimodularity again implies that $\{(0,1),(1,1)\}$ is in $N$. Similarly, if $\{(0,2),(2g+1,0)\}$ is in $N$, then the edges $\{(g,1),(2g+1,0)\}$ and $\{(g-1,1),(2g+1,0)\}$ are in $N$. In each of these cases, $N$ contains a unimodular triangulation of the trapezoidal region with vertices
\[ (a,0), (b,0), (1,1), (g,1) \]
where we take $a=0$ or $a=1$ according to whether the edge $\{(0,2),(1,0)\}$ is in $N$, and similarly we take $b=2g+1$ or $b=2g+2$ according to whether the edge $\{(0,2),(2g+1,0)\}$ is in $N$.  For example, in Figure \ref{f:newton}, we take $a=1$ and $b=2g+2$.

In each case, we see that the dual graph to $N$ is a standard ladder of genus $g$, with vertices $V_1, \dots, V_{g-1}, W_1, \dots, W_{g-1}$. In particular, there are two paths in $T$ from $V_1$ to $W_1$ not passing through any other $V_i$ or $W_i$, and similarly, there are two paths from $V_{g-1}$ to $W_{g-1}$.

Moreover, we claim that for each $i=1, \dots, g-2$, the lattice-lengths of the paths in $\Sigma'$ between $V_i$ and $V_{i+1}$ and between $W_i$ and $W_{i+1}$ are equal. Indeed, for each $i$, write $V_i = (V_{i,1},V_{i,2})$ and $W_i = (W_{i,1},W_{i,2})$. Notice that $V_{i,1} = W_{i,1}$ for each $i$, because each segment $s_i \in N$ is horizontal so its dual in $T$ is vertical. Furthermore, the path from $W_i$ to $W_{i+1}$ in $\Sigma'$ consists of a union of segments of integer slopes, so their lattice-length distance is the horizontal displacement $W_{i+1,1}-W_{i,1}$. Similarly, the distance in $\Sigma'$ from $V_i$ to $V_{i+1}$ is $V_{i+1,1}-V_{i,1} = W_{i+1,1}-W_{i,1}$, as desired. So $\Sigma'$ is a standard ladder of genus $g$, with opposite sides of equal length.
Finally, we apply \cite[Corollary 6.27]{bpr} to conclude that the tropicalization map on the Berkovich analytification $\widehat{X}^{\operatorname{an}}$ of $\widehat{X}$ induces an isometry of the Berkovich skeleton $\Sigma$ of $\widehat{X}^{\text{an}}$ onto $\Sigma'$.  This completes the proof of Theorem~\ref{t:skeleton}.
\end{proof}

In the proof above, we gained a refined combinatorial understanding of the unimodular triangulations of $\Delta_{2g+2,2}$, which we make explicit in the following corollary.

\begin{corollary}
Consider unimodular triangulations of $\Delta_{2g+2,2}$ whose dual complex contains a bridgeless connected subgraph of genus $g$. There are $\binom{3g+3}{g+1}$ such triangulations that use neither the edge $e_1 = \{(0,2),(1,0)\}$ nor the edge $e_2 = \{(0,2),(2g+1,0)\}$, there are $2 \cdot \binom{3g+1}{g}$ triangulations using one of $e_1$ and $e_2$, and there are $\binom{3g-1}{g-1}$ triangulations using both $e_1$ and $e_2$.
\end{corollary}

\begin{proof}
Use the proof of Theorem \ref{t:skeleton}, which characterizes precisely which triangulations can occur, and Lemma \ref{l:trapezoid}.
\end{proof}

\bigskip
\end{document}